\documentclass[11pt]{amsart}
\usepackage{geometry}                
\usepackage{graphicx}
\usepackage{amssymb}
\usepackage{epstopdf}
\usepackage{hyperref}
\newcommand{\R}{\mathbb{R}}

\newcommand{\C}{\mathbb{C}}

\newcommand{\D}{\mathbb{D}}
\newcommand{\Kappa}{\mathcal{K}}

\newcommand{\res}{\operatorname{res}}
\newcommand{\re}{\operatorname{Re}}
\newcommand{\im}{\operatorname{Im}}

\numberwithin{equation}{section}
\numberwithin{figure}{section}

\newcommand{\fracG}{\mathfrak{G}}

\theoremstyle{plain} 
\newtheorem{theorem}{Theorem}[section]
\newtheorem{lemma}[theorem]{Lemma}

\newtheorem{proposition}[theorem]{Proposition}
\newtheorem{conjecture}[theorem]{Conjecture}
\newtheorem{problem}[theorem]{Problem}
\theoremstyle{definition} 
\newtheorem{definition}[theorem]{Definition}
\newtheorem{remark}[theorem]{Remark}
\newtheorem{example}[theorem]{Example}

\title{Complete Embedded Harmonic Surfaces in $\R^3$}

\author
{Peter Connor}
\address{Peter Connor\\Department of Mathematical Sciences\\Indiana University South Bend\\South
Bend\\IN 46634\\USA}
\author{ 
Kevin Li
}
\address{Kevin Li\\Department of Computer Science and Mathematical Sciences\\Penn State Harrisburg\\Middletown, PA 17057\\USA}

\author{
Matthias Weber
}
\address{Matthias Weber\\Department of Mathematics\\Indiana University\\
Bloomington, IN 47405
\\USA}
\thanks{This work was partially supported by a grant from the Simons Foundation (246039 to Matthias Weber)}


\date{\today}

\begin{document}

\begin{abstract}
Embedded minimal surfaces of finite total curvature  in $\R^3$ are reasonably well understood: From far away, they look like intersecting catenoids and planes, suitably desingularized.

We consider the larger class of harmonic embeddings in $\R^{3}$ of compact Riemann surfaces with finitely many punctures where the harmonic coordinate functions are given as  real parts of meromorphic functions. This paper is motivated by two outstanding features of such surfaces:
They can have highly complicated ends, and they still have total Gauss curvature being a multiple of $2\pi$.
This poses the double challenge to construct and classify examples of fixed  total Gauss curvature. Our  results include
\begin{itemize}
\item a classification of embedded harmonic ends of small total curvature,
\item  the   construction of examples of embedded ends of arbitrarily large total curvature,
\item  a classification  of complete embedded harmonic surfaces of small total curvature in the spirit of the corresponding classification of minimal surfaces of small total curvature, and   
\item the largely experimental  construction of  complete embedded harmonic surfaces with non-trivial topology that incorporate some of the new  harmonic ends.
\end{itemize}
\end{abstract}



\maketitle


\section{Introduction}

In this paper we will discuss  harmonic embeddings of punctured compact Riemann surfaces into $\R^3$ whose coordinate differentials admit  meromorphic extensions into the punctures. 

Our original motivation was  to investigate conformally parametrized minimal surfaces of finite total curvature
as a subclass of these more general surfaces, but
preliminary experiments quickly lead to two surprising discoveries: 

\begin{enumerate}
\item In the harmonic context, there exists an abundance of intricate embedded ends, while in the minimal setting of finite total curvature, one only has planar and catenoidal ends.
\item There still holds a simple Gauss-Bonnet formula for complete surfaces, even though the Gauss map does not extend continuously into the punctures.
\end{enumerate}

We have dedicated \cite{clw1} to a (rather technical) proof of this Gauss-Bonnet theorem, and this paper deals with the new examples, some classification results, and several open problems. We begin with a precise definition of the surfaces we will discuss:

Let $X$ be a compact Riemann surface with finitely many distinguished points $p_1,\ldots, p_n$, and $\omega_1$, $\omega_2$, $\omega_3$ be meromorphic 1-forms on $X$ that are holomorphic on $X'=X-\{p_1,\ldots, p_n\}$.
Denote the pole order of $\omega_i$ at $p_j$ by $n^j_i$, and let $n^j = \max\{ n^j_1, n^j_2, n^j_3\}$.

Assume that the periods
\[
 \int_\gamma \left(\omega_1, \omega_2, \omega_3\right)
\]
are  imaginary for all closed cycles $\gamma$ on $X'$. In particular, all residues $\res_{p_j}\omega_i$ are assumed to be real.

 Then the map
\[
f(z) = \re \int^z \left(\omega_1, \omega_2, \omega_3\right)
\]
defines a  harmonic map from $X'$ into $\R^3$. 

The simplest embedded examples of this special type of harmonic maps are graphs of rational functions such as the hyperbolic paraboloid 
\[
f(z)=\re\int^z\left(1,-i,iz\right)
\]
or $f(x,y)=(x,y,xy)$.

The best studied examples of such surfaces are complete minimal surfaces of finite total curvature. By Osserman's theorem \cite{oss2}, such surfaces are defined by meromorphic Weierstrass data on compact Riemann surfaces. However, the geometry of embedded minimal examples is quite limited, as the only complete embedded ends of finite total curvature are catenoidal or planar, giving any such surface the look of  tiered Costa surfaces, as proven to exist in \cite{ww2} and \cite{tr4}. In contrast, we will see that most theorems about minimal surfaces are false for these harmonic surfaces.

However, there is a Gauss-Bonnet theorem for harmonic surfaces, generalizing the 
Gackstatter-Jorge-Meeks formula \cite{gac1,jm1} for minimal surfaces.
In \cite{clw1}, we have proven:

\begin{theorem} 
\label{thm:gaussbonnet}
For harmonic surfaces as above, we have the generalized Gauss-Bonnet formula
\[
\int_{X'} K\, dA - \sum_{i=1}^n \Kappa_i = 2\pi \chi(X)
\]
where 
\[
\Kappa_j = -2\pi (n^j-1)
\]
\end{theorem}

Note also that  in the minimal case, the theorem above is easy to prove, as all complete minimal ends of finite total curvature have a limit tangent plane. 

A consequence of the Gauss-Bonnet theorem is a quantization of the total curvature in integral multiples of $2\pi$, which we will use as a measure of the complexity of the end.

This paper has three main goals:

\begin{enumerate}
\item
Construct examples of embedded harmonic ends of arbitrarily large total curvature.
\item Classify embedded harmonic ends of small total curvature.
\item Classify complete properly embedded harmonic surfaces of small total curvature in the spirit of the corresponding classification of minimal surfaces of small total curvature.
\end{enumerate}

The paper is organized as follows:  In section \ref{sec:low}, we discuss ends of small total curvature. 
While for minimal surfaces, all embedded ends of a given type (and growth rate) are asymptotic to each other, 
this fails for harmonic ends for several reasons. Not only can affine transformations deform an end, but higher order terms can
determine whether an end is eventually embedded or not. Thus while we do have some general statements for ends of  small curvature, the discussion for large curvature is limited to existence and non-existence questions.

The general question to determine whether there are properly embedded ends of a given type is partially addressed in section \ref{sec:high}.  We present two intricate isolated properly embedded examples and give examples of several families of properly embedded ends of arbitrarily high order.

We then turn to the discussion of small total curvature in section \ref{sec:ftc}, where we give a complete classification in case the total curvature is $-2\pi$ or $-4\pi$.

In section \ref{sec:embedded} we will construct complete embedded examples of harmonic surfaces that incorporate some of the ends we have found into surfaces with topology.

Finally, in section \ref{sec:open}, we list open problems that arose naturally in the course of these investigations.

\medskip
This paper has a companion website at \url{http://www.indiana.edu/~minimal/archive/Harmonic/index.html} where one can find more examples with images and Mathematica notebooks. The latter provide additional 
evidence for the sometimes technical regularity and embeddedness claims in form of supporting computations, both using computer algebra and numerics.

\bigskip
Some of the new harmonic ends have an amusing resemblance with classical minimal surface ends. For instance, our end of type $(2,2,3)$ looks at first glance like an Enneper end. It does have the same symmetries, but it is in fact graphical. This resemblance suggests that one can employ
 the well-known Hoffman-Karcher approach (\cite{karch3, hk2}) to construct minimal surfaces with prescribed geometric features:  One first  derives candidate Weierstrass data from divisors of the Gauss map and height differential, and then closes the periods by adjusting parameters. The second step is dramatically simplified because we can close all periods just by adding suitably holomorphic forms to the meromorphic data, which will neither affect the asymptotic behavior of the ends nor the topology of the surface. This allows one to perform virtually all constructions that have been done with Enneper's surface also with the end of type $(2,2,3)$, albeit in an embedded setting.
 
On the other hand, our situation is becoming considerably more complicated when we investigate the more exotic ends, because we have to accommodate their rather intricate asymptotic geometry. Thus the increased flexibility --- more types of embedded ends and no significant period problem --- makes these new examples of harmonic surfaces interesting candidates for applications in geometric modeling and 
design.

\medskip


Many of the examples we construct show how well-known theorems for minimal surfaces in $\R^3$  fail in this more general setting.
For convenience, the following list of theorems points to the relevant examples.

\begin{itemize}
\item Lopez-Ros theorem: The plane and catenoid are the only properly embedded minimal surfaces with finite total curvature and genus zero \cite{lor1}.  See section \ref{subsection:-2pi}.
\item Hoffman-Meeks conjecture: The moduli space $M(k,r)$ of properly embedded minimal surfaces, where $k$ is the genus and $r$ is the number of ends, is empty if $r>k+2$ \cite{hm7}.  See section \ref{subsection:-4pi}.
\item Schoen's 2-end theorem: A properly embedded minimal surface with finite total curvature and two embedded ends must be a catenoid \cite{sc1}.  See section \ref{subsection:-4pi}.
\item Strong Half-Space theorem: Two proper, connected minimal surfaces which don't intersect must be parallel planes \cite{hm10}.  See section \ref{subsection:(0,0,1)}.
\end{itemize}

\section{Classification of embedded ends of low order}
\label{sec:low}

In this section, we will systematically describe the geometry of ends of harmonic surfaces. We begin with preliminary remarks and some notation.

Note that a regular affine transformation can change  the order of the forms $\omega_k$ while not affecting the appearance of the end by much. To obtain a rough classification of ends that is independent of affine modifications, we define:

\begin{definition}
A meromorphic  end  is given by a harmonic  map
 $f:\D^*\to \R^3$ of the form
\[
f(z)=\re \int^z (\omega_1,\omega_2,\omega_3)
\]
where the 1-forms $\omega_i$ are holomorphic in $\D^*$.
We say that two ends $f$ and $\tilde f$  given as above are \emph{affinely equivalent} if there is a regular real affine transformation $A:\R^{3}\to \R^{3}$ such that $\tilde f = A\circ f$.

We say that an end is in \emph{reduced form} if the pole orders $n_k$ of $\omega_k$ at $0$ satisfy $n_1\le n_2\le n_{3}$ and if $(n_1,n_2,n_3)$ is minimal in lexicographic ordering among all affinely equivalent ends. The \emph{type} of an end is the tuple
  $(n_1,n_2,n_3)$ of an affinely equivalent end in reduced form. We then call $n_3$ the \emph{order} of the end.
\end{definition}

With increasing order, our description will become less and less detailed, while the examples will become more and more complicated.

\subsection{Ends of order 1}

Not surprisingly, in this simplest case we have a rather complete classification. We will see that all ends of this type look like the following.

\begin{example}
The prototype of an end of order 1 is given by 
\begin{align*} 
   \omega_1 = {}& 1\, dz\\
    \omega_2 = {}& i\, dz \\
     \omega_3 = {}& \frac1z \, dz \\
\end{align*}
so that 
$f(z)$ is the graph of $\log(|z|)$.
It is easy to create complete surfaces with many ends of order 1, see Figure \ref{figure:urchin}.
\end{example}

Now let's assume that  $f$ parametrizes a harmonic end of order 1 in the punctured disk.
As the harmonic map $f$ is assumed to be single valued, the residues of $\omega_i$ are necessarily real. By applying an affine transformation (that will neither affect embeddedness nor regularity), we can assume that $\omega_1$ and $\omega_2$ are holomorphic at $0$ and $\omega_3 = \frac 1z + O(1)\, dz$. In this case, we say the end is in {\em normal form}. Then we have:

\begin{proposition}
Let $f:\D\to \R^3$ be a meromorphic end of order 1 in normal form with
\begin{align*} 
   \omega_1 = {}& \left(a_1+b_1 i + O(1)\right)\, dz\\
    \omega_2 = {}& \left(a_2+b_2 i + O(1)\right)\, dz \\
     \omega_3 = {}& \left(\frac1z + a_3+b_3 i + O(1)\right)\, dz \\
\end{align*}
If  $a_1+b_1 i$ and $a_2+b_2 i$ are independent over $\R$, this end is properly embedded.
\end{proposition}

\begin{proof}
After an affine transformation and a holomorphic change of coordinates (which does not affect harmonicity) in the domain $\D$, we can assume that  

\begin{align*} 
   \omega_1 = {}& \, dz \\
    \omega_2 = {}& \left( i+c z+ O(1)\right) \, dz \\
     \omega_3 = {}&\left( \frac1z  + O(1)\right) \, dz \\
\end{align*}

The independence assumption allows us to make the zeroth order terms of $\omega_1$ and $\omega_2$ equal to 1 and $i$, respectively.

This defines indeed a proper immersion near $0$.  For embeddedness, we can  use (by the inverse function theorem) a (non-holomorphic) diffeomorphism in the domain that will change the forms to 

\begin{align*} 
   \omega_1 = {}& dz \\
    \omega_2 = {}&  i\, dz \\
\end{align*}

with $\omega_3$ being a closed differentiable complex-valued 1-form in the punctured disk. This change makes the parametrization non-harmonic, but 
 identifies the end as a graph over a punctured disk in the $x_1x_2$-plane.
 \end{proof}

If the two leading coefficients of $\omega_1$ and $\omega_2$ are dependent over $\R$, the situation becomes complicated. One can easily construct examples that are neither embedded nor regular. But there are also cases where the end is still properly embedded:

\begin{example}\label{ex:subtle1}
The  end of order 1  given by 
\begin{align*} 
   \omega_1 = {}& \left(1+2(1+i)z+3iz^2\right) \, dz\\
    \omega_2 = {}& \left(1+2iz\right)\, dz \\
     \omega_3 = {}& \frac1z \, dz \\
\end{align*}
is properly embedded. To see this, consider the affinely equivalent end 
\begin{align*} 
   \omega_1 = {}& \left(2z+3iz^2\right) \, dz\\
    \omega_2 = {}& \left(1+2iz\right)\, dz \\
     \omega_3 = {}& \frac1z \, dz \\
\end{align*}
It suffices then to show that the $x_3$-level curves $t\mapsto \left(\cos(2t)-r\sin(3t), \cos(t)-r\sin(2t)\right)$
are regular and embedded. We leave the details to the reader.

%

More generally, consider the ends given by
\begin{align*} 
   \omega_1 = {}& \, dz\\
    \omega_2 = {}& \left(z +c i z^n \right)\, dz \\
     \omega_3 = {}& \frac1z \, dz \\
\end{align*}

and denote the resulting parametrization by $f$. Then the second coordinate of $f_x\times f_y$ is equal to $-y/(x^2+y^2)$ so that the end
is regular away from $y=0$. For $y=0$ we have
\[
f_x\times f_y = c\left( x^{n-1}, 0, x^n\right)
\]
so that the end is regular in a punctured disk if and only if $c\ne 0$.

\end{example}

\subsection{Ends of order 2}
In this section, we will recover the two embedded minimal ends and one new type of harmonic end.
As above,  we first derive normal forms of such ends. Given an end of order 2 but not less, one of the forms (say $\omega_3$) has a pole of order 2. If all forms have poles of order 2, the coefficients of the $1/z^2$ are dependent over $\R$. After an affine change of coordinates,  we can assume that $\omega_1$ has a pole of order at most 1. If there are two forms with a pole of order 1, then another affine change can be used to make one of them holomorphic. We say that an end is in reduced form when all these changes have been applied. Denote the type of the end by $(n_1,n_2,n_3)$, where $n_i$ is the order of the pole of $\omega_i$. Then the following types of ends of order 2 in reduced form can occur:

\begin{theorem}
\label{thm:order2}
After an affine change of coordinates, the following end types of order 2 can occur, with the indicated properties:
\begin{enumerate}
\item
$(1,2,2)$: catenoidal --- always properly embedded
\item
 $(0,2,2)$: planar --- always properly embedded
 \item 
 $(0,1,2)$ ---  can be properly embedded 
 \item 
 $(0,0,2)$ --- never proper
\end{enumerate}
\end{theorem}

Note again that all residues of the first order poles must be real, but the coefficients of the second order poles are not affected. Note that the theorem lists all possibilities of ends of order two in reduced form. We will now prove this theorem by discussing these cases individually.

\begin{lemma}
Any catenoidal or planar end  is always properly embedded.
\end{lemma}

\begin{proof}
Suppose first we are given a meromorphic end of type $(1,2,2)$ where the residue of $\omega_1$ at $0$ can be assumed to be $1$. By applying an affine transformation, we can also assume the the residues of $\omega_2$ and $\omega_3$ at $0$ are both 0. Thus we can write
\begin{align*} 
   \omega_1 = {}&  \left(\frac1z +O(1)\right)\, dz \\
    \omega_2 = {}&  \left( \frac {a_2+b_2 i}{z^2} +O(1)\right)\, dz \\
     \omega_3 = {}& \left( \frac {a_3+b_3 i}{z^2} +O(1)\right)\, dz \\
\end{align*}

If $a_2+b_2 i$ and $a_3+b_3 i$ are independent over $\R$, we can apply another affine transformation in the $x_2x_3$ plane to obtain
\begin{align*} 
   \omega_1 = {}&  \left(\frac1z +O(1)\right)\, dz \\
    \omega_2 = {}&  \left( \frac {1}{z^2} +O(1)\right)\, dz \\
     \omega_3 = {}& \left( \frac {i}{z^2} +O(1)\right)\, dz\\
\end{align*}
Using a smooth change of coordinate in $\D^*$, we can finally achieve that
\begin{align*} 
   \omega_1 = {}& smooth \\
    \omega_2 = {}&   \left(\frac {1}{z^2} \right)\, dz \\
     \omega_3 = {}&  \left(\frac {i}{z^2}  \right)\, dz\\
\end{align*}
where $\omega_1$ is a smooth closed 1-form without periods in $\D^*$. This proves that the end is a graph over the complement of a large disk in the $x_2x_3$-plane. It is also clearly regular and proper.

Now assume that $a_2+b_2 i$ and $a_3+b_3 i$ are dependent over $\R$. Then an affine change of coordinates reduces this type of end further to type $(1,0,2)$, which will be discussed below as an end of type $(0,1,2)$.

The planar case $(0,2,2)$ is treated in the same way.
\end{proof}

\begin{figure}[h]
	\centerline{ 
	\includegraphics[width=2.5in]{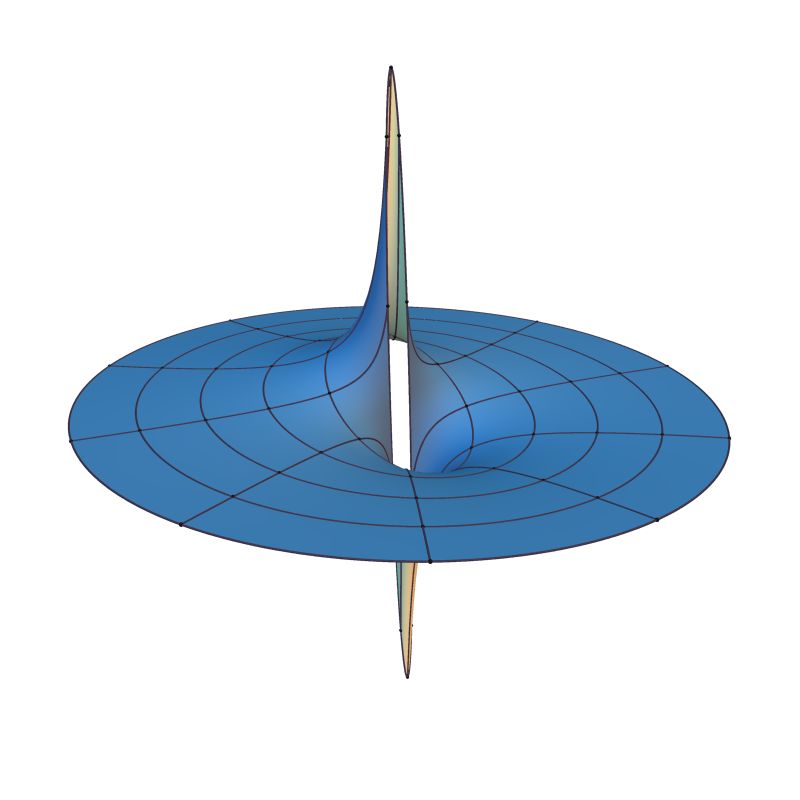}
	}
	\caption{Embedded but non-proper end of type  $(0,0,2)$}
	\label{figure:(0,0,2)}
\end{figure}

The case $(0,0,2)$ is a special case of the more general
\begin{lemma}
\label{lemma:00n}
For $n\ge2$, an  end of type $(0,0,n)$ is never proper.
\end{lemma}
\begin{proof}
We use the normalization of 1-forms via holomorphic coordinate changes (Proposition 4.1 from \cite{clw1}) which implies  that we can assume that $\omega_1$ and $\omega_2$ are holomorphic, while $\omega_3 = \left(z^{-n}+a/z)\right)\, dz$ with some real number $a$.

We have to show that there is a sequence $z\to0$ such that $f(z)$ remains bounded. As the first two coordinate 1-forms are holomorphic at $0$, we will do so for $f_3(z)$.
Integrating $\omega_3$ in polar coordinates $z=e^{r+i t}$  gives
\[
f_3(r,t) =a r-\frac{1}{n-1} e^{(1-n)r} \cos((n-1)t)
\]
For $r\ll 0$, this can be made 0 by solving for $t$. 
\end{proof}

We finally discuss ends of type $(0,1,2)$. Here, the embeddedness will depend in a  subtle way  on the higher order terms in the Laurent expansion of the coordinate 1-forms. 

 By making a holomorphic coordinate change and by scaling the surface, if necessary, in the coordinate directions,  we can assume that 
 \begin{align*} 
   \omega_1 = {}& \left(e^{i\phi} +O(z)\right) \, dz \\
    \omega_2 = {}& \frac {1}{z} \, dz\\
     \omega_3 = {}& \left(e^{i\psi}\frac {1}{z^2} +O(1)\right)\, dz
\end{align*}

\begin{figure}[h]
	\centerline{ 
	\includegraphics[width=2.5in]{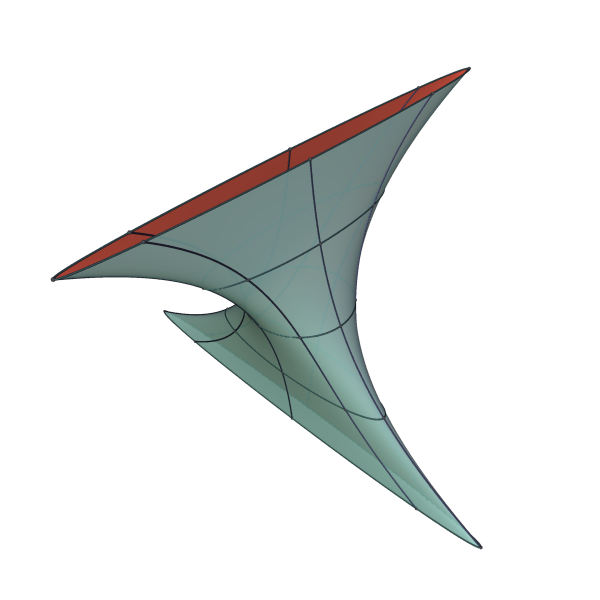}
	}
	\caption{Embedded sphere with  two $(0,1,2)$ ends}
	\label{figure:(0,1,2)x2}
\end{figure}

\begin{example}
A prototype of an end of type $(0,1,2)$ is  given by
 \begin{align*} 
   \omega_1 = {}& dz \\
    \omega_2 = {}&  \frac {1}{z} \, dz\\
     \omega_3 = {}&\frac i{z^2}\, dz
     \end{align*}
which integrates in polar coordinates  to 
\[
f(r,t) = \left( e^r \cos(t), r, e^{-r}\sin(t)\right)
\]
which is clearly regular and properly embedded. Geometrically, the end at  $z=0$ lies in an arbitrarily thin half slabs $x_2>r$ and $|x_3|<e^{-r}$. Note that the prototype above has two ends of type $(0,1,2)$ at $0$ and $\infty$ that are asymptotic to half slabs that are turned by $90^\circ$ against each other. While one can affinely  change the angle between the two ends, the maximum principle prevents the ends from becoming parallel.
\end{example}

More generally, we conjecture 
that in the case  that $\phi\ne -\psi$ and $\phi\ne \pi-\psi$, the surface is embedded as well, because the level curves in the planes $x_2=const$ appear nearly elliptical.

However, when say  $\phi= -\psi$, the situation becomes more subtle. 

\begin{example}
\label{ex:subtle2}
For simplicity and to illustrate the difficulties, let's compare
 \begin{align*} 
   \omega_1 = {}& 1\, dz   \\
    \omega_2 = {}&  \frac {1}{z}\, dz \\
     \omega_3 = {}&\left( \frac {1}{z^2} +i \right)\, dz
\end{align*}

to

 \begin{align*} 
   \omega_1 = {}& 1 \, dz \\
    \omega_2 = {}&  \frac {1}{z} \, dz\\
     \omega_3 = {}& \left(\frac {1}{z^2} + i z\right)\, dz
\end{align*}

The first  one is  easily seen to be embedded (albeit barely), while the second one satisfies $f(i t) = f(- it)$ for all $t\in\R$.
Thus it appears that embeddedness will depend in a rather subtle way
on the coefficients.
\end{example}

%


\subsection{Ends of order 3}
\label{sec:order3}
There are nine different types of ends of order 3, four of which can be embedded. As in the case of ends of type $(0,1,2)$, the embeddedness will depend in subtle ways on higher order terms in the Laurent expansion. It will turn out that the embedded ends of this section have immediate generalizations to embedded ends of higher order.

We begin with the usual normalization:
 Given an end of order 3 but not less, one of the forms (say $\omega_3$) has a pole of order 3. If all forms have poles of order 3, the coefficients of the $1/z^3$ are dependent over $\R$. Thus we can assume that $\omega_1$ has a pole of order at most 2. Denote the type of the end by $(n_1,n_2,n_3)$, where $n_k$ is the order of the pole of $\omega_k$. Then we have:

\begin{theorem}
\label{thm:order3}
The following types of reduced properly embedded ends of order 3 can occur:
An end of type 
 $(2,2,3)$ is always properly embedded. Ends of type 
 $(1,2,3)$,  $(0,2,3)$, or   $(0,1,3)$ can be properly embedded under suitable conditions.
 All other possible ends are of one of the following types and never properly embedded: 
 $(2,3,3)$, $(1,3,3)$, $(0,3,3)$, $(0,0,3)$.
\end{theorem}

\begin{proof}
This theorem will follow from the sequences of propositions as follows:
\begin{itemize}
\item $(2,2,3)$: In Proposition \ref{thm:22n} we show that all ends of type $(2,2,n)$ are properly embedded for $n\ge3$.
\item $(1,2,3)$:  In Proposition \ref{thm:12n} we show that there are properly embedded ends of type $(1,2,n)$ for $n\ge3$.
\item $(0,2,3)$:  In Proposition \ref{prop:(0,2,3)} we show that there are properly embedded ends of type $(0,2,3)$ for $n\ge3$.
\item $(0,1,3)$:  In Proposition \ref{thm:(0,1,n)} we show that there are properly embedded ends of type $(0,1,n)$ for $n\ge3$.
\item $(0,0,3)$: In Lemma \ref{lemma:00n} we show that for $n\ge2$, ends of type $(0,0,n)$ are never proper.
\item $(2,3,3)$, $(1,3,3)$, $(0,3,3)$: In Lemma \ref{lm.(j,k,k)} we show that reduced ends of type $(n_1,n_2,n_2)$ with $n_2\ge 3$ cannot be embedded.
\end{itemize}
\end{proof}

The following general lemma implies in particular that the ends of type $(2,3,3)$, $(1,3,3)$, and $(0,3,3)$
are never embedded:

\begin{lemma}
\label{lm.(j,k,k)}
A reduced end of type $(n_1,n_2,n_2)$ with  $n_2>2$ cannot be embedded.
\end{lemma}


\begin{proof}
The proof is an adaptation of the well known argument that shows that the only complete embedded minimal ends of
finite total curvature are planar or catenoidal:
We can assume that the coefficients  of $1/z^{n_2}$ in $\omega_2$ and $\omega_3$ are $1$ and $i$, respectively --- otherwise an affine transformation could reduce the order of one end. 

This shows that the image of any simple closed curve around 0 near 0 projects onto the $x_2x_3$-plane as a curve with winding number $n_2-1>1$ with respect to $0$.

Thus, if we intersect the end with a cylinder about the $x_1$-axis of sufficiently large radius, the  intersection curve will wind $n_2-1$ times about the $x_1$-axis, thus can't be embedded.
\end{proof}

Note that this Lemma also applies to ends of type say $(4,3,3)$. It is thus rather surprising that we can show that existence of embedded ends where all $n_k\ge 3$, see Example \ref{ex:346}.

By Lemma \ref{lemma:00n}, an end of type $(0,0,3)$ is never proper. This, together with Lemma \ref{lm.(j,k,k)}, proves the last claim of the theorem.

We now turn to ends of type $(2,2,n)$ for $n\ge2$:

\begin{example}
The harmonic surfaces given by
 \begin{align*} 
   \omega_1 = {}& \, dz  \\
    \omega_2 = {}&  i\, dz \\
     \omega_3 = {}& z^n \, dz
\end{align*}
are graphs over the $xy$-plane with an end  of type $(2,2,n+2)$ at $\infty$. 
In the case $n=3$, this is the hyperbolic paraboloid.
\end{example}

Ends of type $(2,2,3)$ are always embedded as graphs. More generally, we have:

\begin{proposition}\label{thm:22n} An end of type $(2,2,n)$ is always properly embedded.
\end{proposition}
\begin{proof}
We place the end at $\infty$ and show that it is graphical.
In general, the leading coefficients of $\omega_1$ and $\omega_2$ are independent over $\R$, as otherwise the type of the end could be reduced.  Furthermore, we can assume that at most $\omega_2$ has a simple pole. Then there is a  holomorphic change of coordinate after which $\omega_1=dz$ and $\omega_2=\left( i+a_1/z +a_2/z^2+\cdots\right)\, dz$ in the complement of a disk. This given, it is easy to see that for $|z|$ large enough, the map $z\mapsto \re\int^z (\omega_1, \omega_2)$ is a diffeomorphism onto its image that is proper.
\end{proof}

It is easy to use these ends to construct higher genus surfaces akin to the Chen-Gackstatter surface and Thayer's generalizations \cite{tha3,sa1,ww1}, as done in section \ref{sec:surf223}.

Next we will discuss ends of type $(1,2,3)$ as part of a series of properly embedded spheres with an end of type $(1,2,n)$ and a second end of type $(0,0,1)$:

\begin{figure}[h]
	\centerline{ 
	\includegraphics[width=2.5in]{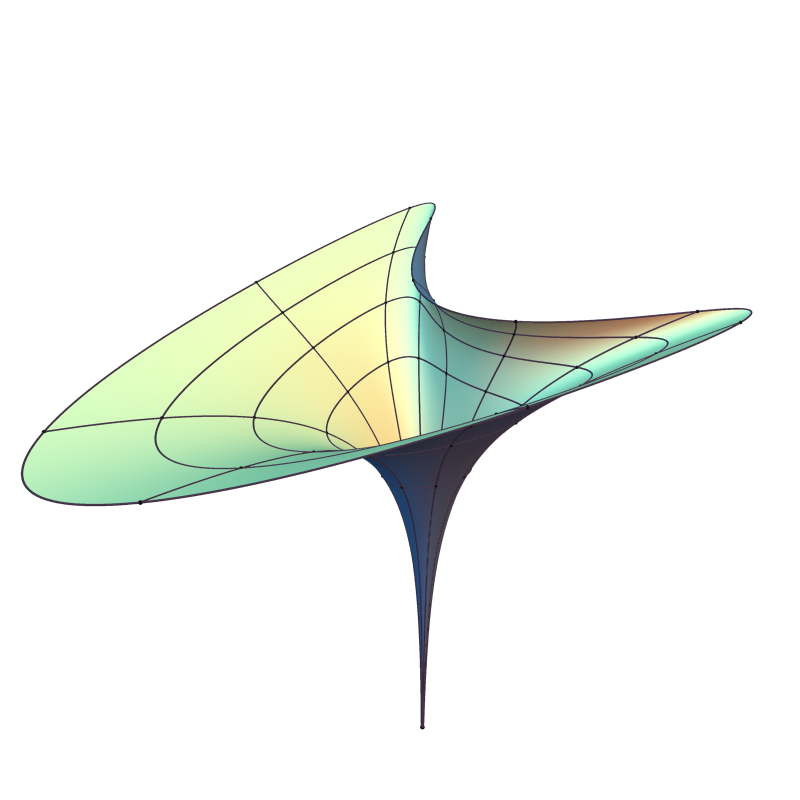}
	\includegraphics[width=2.5in]{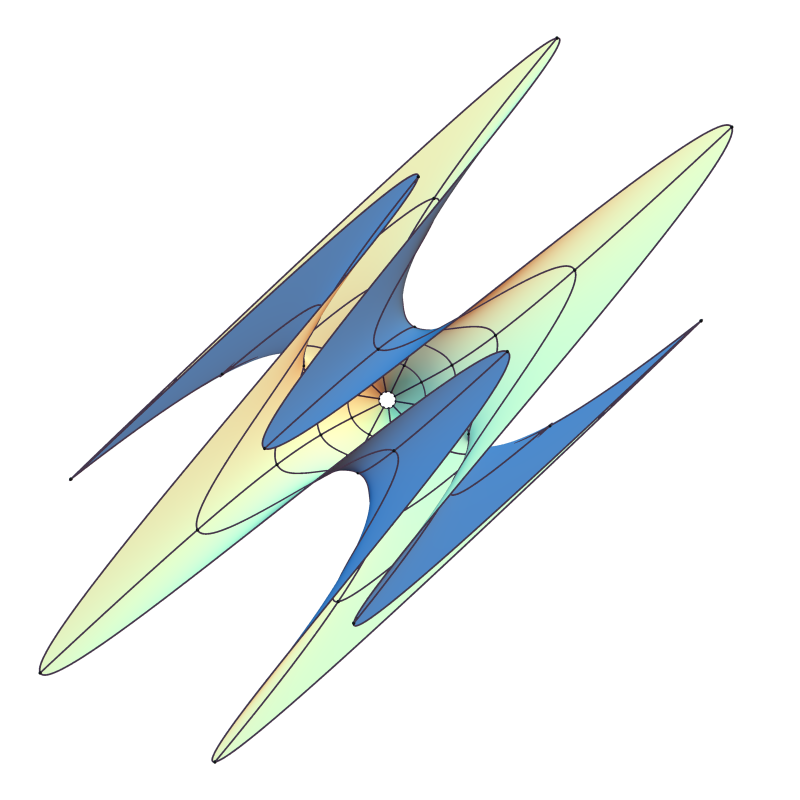}
	}
	\caption{Embedded spheres with  one $(1,2,3)$- resp. $(1,2,8)$-end}
	\label{figure:End(1,2,n)}
\end{figure}

\begin{proposition}\label{thm:12n}
The surfaces  given for $n\ge 3$ on $\C^*$ by
 \begin{align*} 
   \omega_1 = {}& \frac1z \, dz  \\
    \omega_2 = {}&   \, dz \\
     \omega_3 = {}&\left( i+z^{n-2}\right) \, dz\ .
\end{align*}
are complete and properly embedded.
They have at  $0$ an end  of type $(1,0,0)$ and at $\infty$ an end of type  $(1,2,n)$.
\end{proposition}
\begin{proof}
The statement about the order of the ends is clear. To check properness, we compute in polar coordinates
\[
f(r,t) = \left(r \cos(t), \log(r), \frac1{n-1} r^{n-1}\cos((n-1)t) - r \sin(t)\right)
\]
and observe that the second coordinate will tend to $\infty$ for $r\to 0$ or $r\to \infty$.
For regularity, we have
\[
f_r\times f_t = \left(
-\cos(t)-r^{n-2} \sin((n-1)t), r+r^{n-1}\sin((n-2)t), \sin(t)
\right)
\]
If $f_r\times f_t =0$, then $\sin(t)=0$, but then $f_r\times f_t =(\pm 1,r,0)$.

To check embeddedness, note that $f(r,t)$ determines $r$ via the second coordinate. Thus the first coordinate
determines $\cos(t)$. As $\cos((n-1)t)$ can be expressed as a polynomial in $\cos(t)$ for any integer $n$, the last coordinate can be used to also determine $\sin(t)$, and thus $t \pmod {2\pi}$
\end{proof}

\begin{figure}[h]
	\centerline{ 
	\includegraphics[width=2.5in]{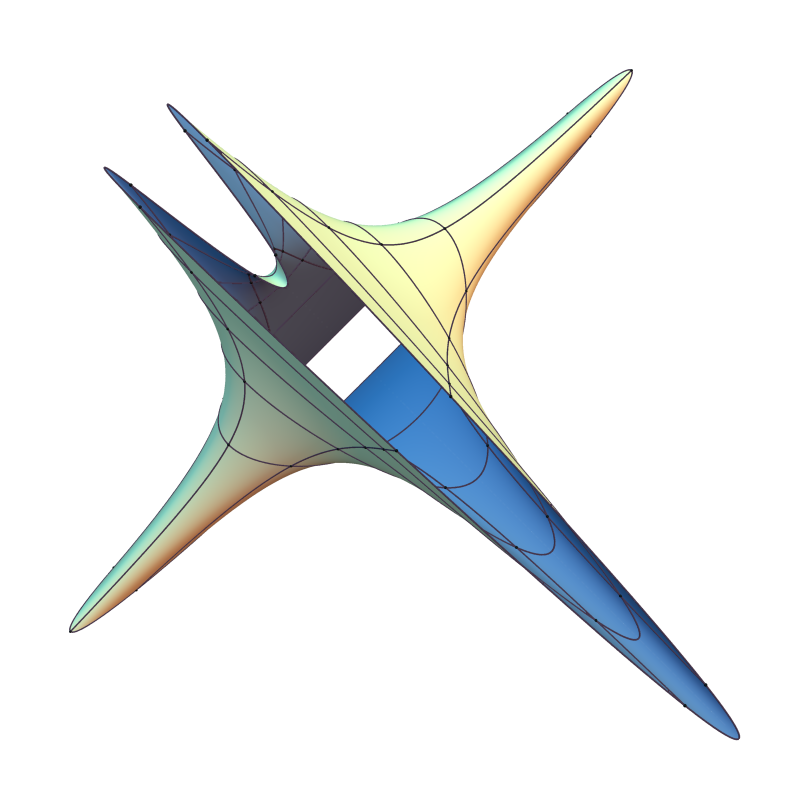}
	\includegraphics[width=2.5in]{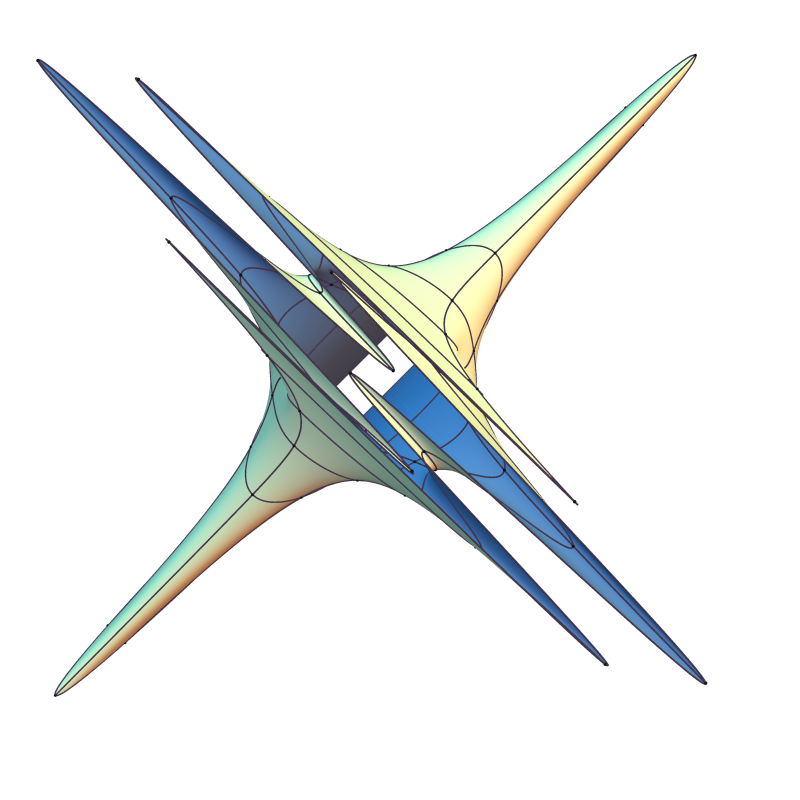}
	}
	\caption{Embedded spheres with  one $(0,1,3)$- resp. $(0,1,6)$-end}
	\label{figure:End(0,1,n)}
\end{figure}

Our examples of embedded ends of type $(0,1,3)$ are derived as a special case of an analogous series:

\begin{proposition}\label{thm:(0,1,n)}
The surfaces  given for $n\ge 3$ on $\C^*$ by
 \begin{align*} 
   \omega_1 = {}& \frac1{z^2} \, dz  \\
    \omega_2 = {}&  \frac1{z} \, dz \\
     \omega_3 = {}& \left(i+z^{n-2}\right) \, dz\ .
\end{align*}
are complete and properly embedded.
They have at  $0$ an end  of type $(0,1,2)$ and at $\infty$ an end of type  $(0,1,n)$.
\end{proposition}
\begin{proof}
Analogous to the proof of the previous theorem.
\end{proof}

Ends of type $(0,2,3)$ are similar to the ones of type $(1,2,3)$.  They also stay in thinner and thinner slabs, but require more space. An embedded example with two such ends is given by

\begin{proposition}\label{prop:(0,2,3)}
The surface given by
\begin{align*}
\omega_1&=\left( -z+\frac1{z^3} \right)\, dz
\\
\omega_2&=\frac{i}{z^2}\, dz
\\
\omega_3 &=  dz
\end{align*}
is properly embedded and has two ends of type $(0,2,3)$ at $0$ and $\infty$.
\end{proposition}
We omit the proof, as it  involves no new ideas nor causes any difficulties. 

\begin{figure}[h]
	\centerline{ 
	\includegraphics[width=2.5in]{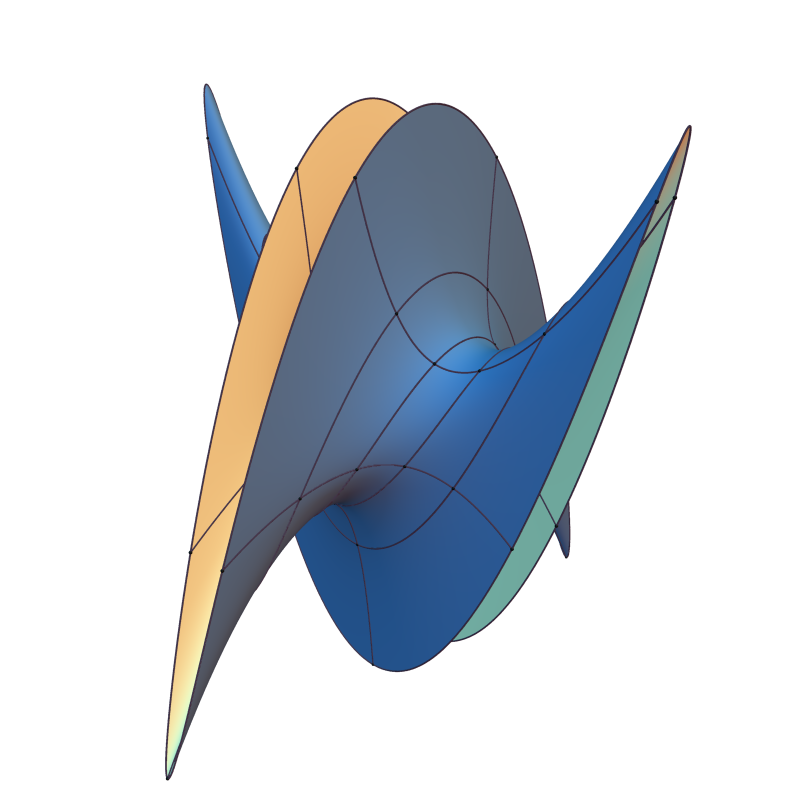}
	\includegraphics[width=2.5in]{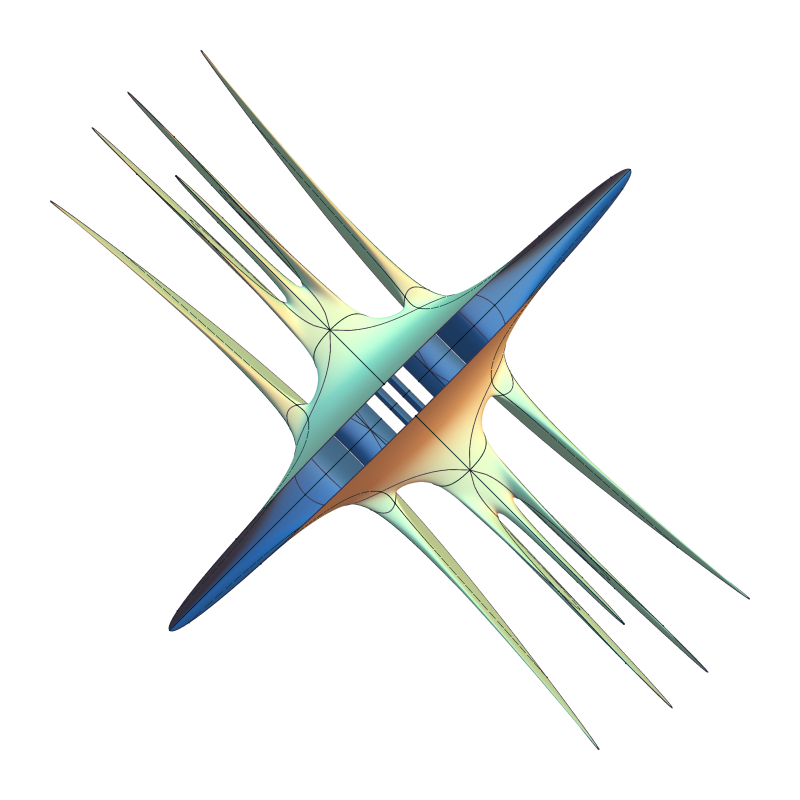}
	}
	\caption{Embedded spheres with  two and six ends of type $(0,2,3)$}
	\label{figure:End(0,2,3)}
\end{figure}

This concludes our discussion of ends of order $3$. At this point, it is completely open whether a similar classification of ends of higher of order can be achieved. 
We don't even know in general whether for a given type there exists a complete and properly embedded end of this type.

So far, we have seen that there are properly embedded ends of arbitrarily high order.
In the following section, we will showcase some extraordinarily complicated ends of high order that
resulted as part of our efforts to answer the above question.

\section{Families of ends of higher order}
\label{sec:high}

A key question is what types of {\em embedded} ends are possible for harmonic surfaces defined by meromorphic 1-forms.

In this section we will present three families of such ends. In fact, most of the ends in this section are given as the ends of complete, embedded surfaces. We also give two embedded ends (again as complete surfaces) that, while in the same spirit, do not fall in any of these three families.  Finally, we end the section with a large class of fairly complicated ends that we conjecture to be embedded.

\subsection{Ends of type \texorpdfstring{$(2,3,n)$}{(2,3,n)} and \texorpdfstring{$(2,3,2n)$}{(2,3,2n)}}
\label{sec:(2,3,n)}
For simplicity in the proofs, we place the ends at infinity.
\begin{proposition} \label{thm:(2,3,n)}
The surfaces given on $\C$ by 
\begin{align*}
\omega_1&= 1 \, dz
\\
\omega_2&=z\, dz
\\
\omega_3 &= \left(z^{n}+i\right) \, dz\qquad n\ge2
\end{align*}
and
\begin{align*}
\omega_1&= i \, dz
\\
\omega_2&=i z\, dz
\\
\omega_3 &= \left( z^{2n}+1\right) \, dz \qquad n\ge1
\end{align*}
are regular proper embeddings.
\end{proposition}

\begin{proof}
We only prove this result for the second set of surfaces.  The proof for the first set is slightly easier.

The last coordinate of $f_x\times f_y$ is equal to $\im(\omega_1 \overline {\omega_2}) = -y$ so that 
$f$ can only have singularities when $y=0$, that is $z$ is real.  The second coordinate of $f_x\times f_y$ is equal to $-\im(\omega_1 \overline {\omega_3})$, which evaluates $1+z^{2n}$ when $z$ is real. So $f$ is regular everywhere.

To see that $f$ is an embedding, we compute
\[
f(x,y) = \left(-y, -xy, \frac{\left(x^2+y^2\right)^{n+\frac{1}{2}} \cos \left((2n+1)
   \arctan\left(y/x\right)\right)}{2n+1}+x\right)
\]
If $f(x_1,y_1) = f(x_2,y_2)$, we then have $y_1=y_2$ from the first coordinate. If these are nonzero, then $x_1 = x_2$ from the second coordinate.  If $y=0$, then the third coordinate simplifies to $\frac{1}{2n+1}x^{2n+1} + x$ which is a strictly increasing function.  This implies that $x_1 = x_2$ and so $f$ is one-to-one.

To see that $f$ is proper, we consider a sequence $(x,y)\to\infty$.
If $y$ is unbounded, so is the first coordinate of $f(x)$. Thus we can assume that $y$ is bounded. Thus $x\to\infty$. If $y$ is bounded away from $0$, the second coordinate of $f(x)$ converges to $\infty$. Thus we can assume that $y\to 0$. But then the last coordinate is approximately $\frac{1}{2n+1} x^{2n+1} + x$ which is unbounded. 
\end{proof}

\begin{figure}[h]
	\centerline{ 
	\includegraphics[width=2.5in]{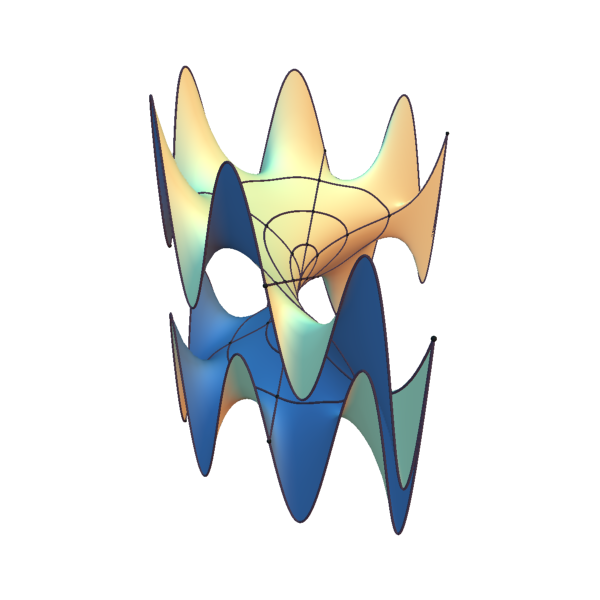}
		\includegraphics[width=2.5in]{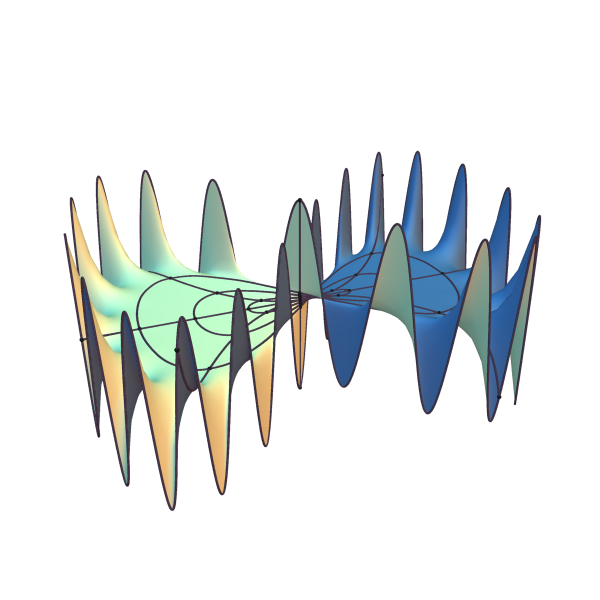}
	}
	\caption{Embedded sphere with  one $(2,3,n)$-end- and $(2,3,2n)$-end}
	\label{figure:End(2,3,n)}
\end{figure}

\subsection{Ends of type \texorpdfstring{$(2,4,2n)$}{(2,4,2n)}}

The ends of  this section are forming surfaces with a very sharp ``edge'' where two ``sheets'' fold very close
to each other. 


\begin{proposition}
The surfaces given on $\C$ by 
\begin{align*}
\omega_1&= 1 \, dz
\\
\omega_2&=z^2\, dz
\\
\omega_3 &= \left(z^{2n}+i\right) \, dz
\end{align*}
for $n \geq 2$ are regular embeddings  that are not proper.
\end{proposition}

%
%

The proof is similar to the proof of Proposition \ref{thm:(2,3,n)}.  Unfortunately, these surfaces are not proper, as it is possible to choose an unbounded sequence of points that will approach the origin in the image.  More specifically, the surface is the image of 
\[
f(x,y) = \left(x, \frac13x(x^2-3y^2), \frac{\left(x^2+y^2\right)^{n+\frac{1}{2}} \cos \left((2 n+1)
  \arctan\left(y/x\right)\right)}{2 n+1}-y\right).
\]

Consider the curve $f\left(\frac{(-1)^{n}}{y^{2n-1}}, y\right)$ as $y \rightarrow \infty$.  Clearly the first two coordinates of $f$ converge to 0 if $n \geq 2$.  Let us consider the third term
\begin{align*}
\cos \left((2n+1) \arctan\left(\frac{y}{x} \right) \right) & = (-1)^{n} \sin \left( (2n+1) \frac{\pi}{2} - (2n+1) \arctan \left( (-1)^{n} y^{2n} \right) \right) \\
& = (-1)^{n} \left[ (2n+1) \frac{\pi}{2} - (2n+1) \arctan\left( (-1)^{n} y^{2n} \right)  + \right.\\
& \quad \quad \quad\left. \frac{1}{6} \left( (2n+1) \frac{\pi}{2} - (2n+1) \arctan \left( (-1)^{n} y^{2n} \right) \right)^{3} - \dots\right]\\
& = (-1)^{n} \left[(-1)^{n} (2n+1) y^{-2n}  + O(y^{-4n}) \right] \\
& = (2n+1) y^{-2n} + O(y^{-4n})
\end{align*}
where we have used $\arctan x + \arctan \frac{1}{x} = \frac{\pi}{2}$.  So, the third coordinate is
\[
\frac{\left(x^2+y^2\right)^{n+\frac{1}{2}} \cos \left((2 n+1) \arctan\left(y/x\right)\right)}{2 n+1}-y  = O(y^{-4n})
\]
which shows that the curve is bounded.

\begin{figure}[h]
	\centerline{ 
		\includegraphics[width=2.5in]{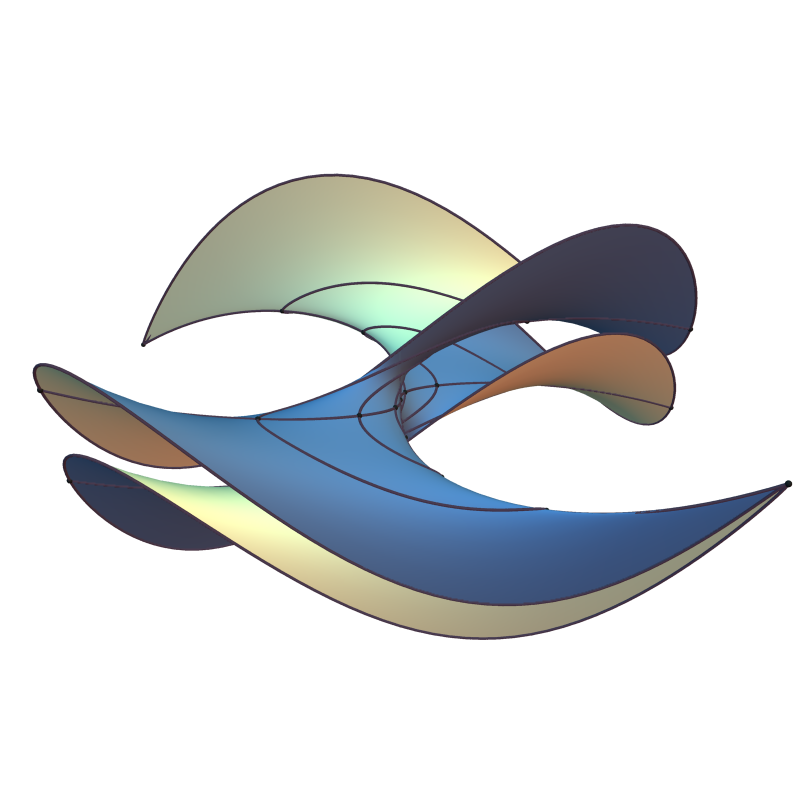}
		\includegraphics[width=2.5in]{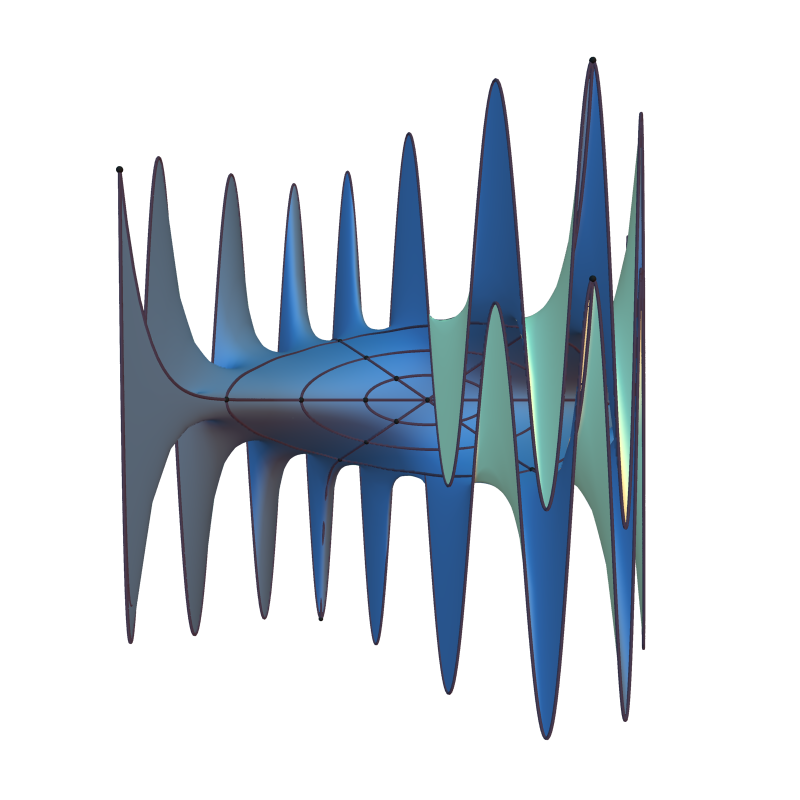}
	}
	\caption{Ends of type $(2,4,6)$ and $(2,4,20)$}
	\label{figure:24n}
\end{figure}

\subsection{An End of type \texorpdfstring{$(2,5,8)$}{(2,5,8)}}

Besides the series $(0,1,n)$, $(1,2,n)$, $(2,2,n)$, $(2,3,n)$ and $(2,3,2n)$, we have two rather complicated complete, properly embedded examples that might or might not be part of a more general series. One of them is the end of type $(2,5,8)$.

\begin{proposition}
The surface given on $\C$ by 
\begin{align*}
\omega_1&= i \, dz
\\
\omega_2&=i z^3\, dz
\\
\omega_3 &= \left(z^6+1\right)\, dz
\end{align*}
is a regular proper embedding.
\end{proposition}
\begin{proof}
To check this, we compute
\[
f(x,y)=\left(-y,-x^3y+xy^3,x+\frac{x^7}{7}-3x^5y^2+5x^3y^4-xy^6\right)
\]

Now, let us assume that $f(x_1, y_1) = f(x_2,y_2)$.  By the first coordinate, we have $y_1 = y_2$, so we may omit the subscripts on $y$.  Then, from the second coordinate,
\[
(x_2-x_1)y(x_1^2+x_1x_2+x_2^2-y^2)=0
\]
and so $x_1=x_2$, $y=0$, or $y^2=x_1^2+x_1x_2+x_2^2$.  If $y=0$ then, from the third coordinate reduces to $\frac{1}{7}x^7 + x$ which is strictly increasing, meaning $x_1 = x_2$.  So let us consider the last case.  Note that $y^2$ is a bivariate homogeneous polynomial of degree 2.

Let $h(x,y) = \frac{x^7}{7} - 3x^5 y^2 + 5x^3y^4- xy^6+x$.  Then
\begin{align*} 
h(x_1,y) - h(x_2,y) & =  \frac{x_1^7}{7} - \frac{x_2^7}{7} - 3x_1^5 y^2+ 3x_2^5y^2 + 5x_1^3y^4 - 5x_2^3y^4- x_1y^6+x_2y^6+x_1-x_2\\
& = \frac{1}{7} (x_1-x_2)\left(7+\sum_{i=0}^{6} x_{1}^{6-i}x_2^{i} -21y^2 \sum_{i=0}^{4}x_{1}^{4-i}x_{2}^{i}+35y^4 \sum_{i=0}^{2} x_{1}^{2-i}x_{2}^{i}-7y^6\right).
\end{align*}
Let us examine the last term in the factorization:
\[
k(x_1,x_2,y) = \sum_{i=0}^{6} x_{1}^{6-i}x_2^{i} -21y^2 \sum_{i=0}^{4}x_{1}^{4-i}x_{2}^{i}+35y^4 \sum_{i=0}^{2} x_{1}^{2-i}x_{2}^{i}-7y^6.
\]
By the observation on $y^2$, this is a homogeneous polynomial of degree 6.  Without loss of generality, let $x_2 = mx_1$.  Then $y^2 = (1+m+m^2)x_1^2$ and
\begin{align*} 
k(x_1,mx_1,y) & = \left(\sum_{i=0}^{6} m^i -21 (1+m+m^2) \sum_{i=0}^{4} m^{i} +35 (1+m+m^2)^2 \sum_{i=0}^{2} m^{i} - 7(1+m+m^2)^3 \right) x_{1}^{6} \\
& = \begin{cases} \left( \frac{1-m^7}{1-m} -21 \frac{1-m^3}{1-m} \frac{1-m^5}{1-m} +28 \left(\frac{1-m^3}{1-m}\right)^3 \right) x_1^6 & \text{if $m\neq 1$}\\
448x_1^6 & \text{if $m=1$} \end{cases}
\end{align*}
In either case, the coefficient of $x_1^6$ is positive.  To see this for $m\neq 1$, we first expand part of the coefficient
\[
\frac{1-m^7}{1-m} + 7\frac{1-m^3}{1-m} (m^4+5m^3+9m^2+5m+1).
\]
Clearly, the first term is positive, so we just need to check that the second term is also positive.  Clearly, this holds form $m=0$, so let us assume that $m\neq0$.  Let $z = m + \frac{1}{m}$ and write
\[
m^4+5m^3+9m^2+5m+1 = m^{2} \left( m^2 + 5m+ 9 + \frac{5}{m} + \frac{1}{m^2} \right) = m^2 \left( z^2 + 5z+7\right).
\]
Both terms are positive and so we are done.

Hence, $k(x_1,x_2,y)\geq 0$ for $y^2 = x_1^2 + x_1 x_2+x_2^2$. Therefore from 
\[
h(x_1,y)-h(x_2,y) = \frac{1}{7} (x_1 - x_2) (7 + k(x_1,x_2,y)) = 0,
\]
we have $x_1 = x_2$ and $f$ is one-to-one.

The last coordinate of $f_x\times f_y$ is $\im(\omega_1\overline{\omega_2})=-y(3x^2-y^2)$.  Thus, $f$ has a singularity only when $y=0$ or $y=\pm\sqrt{3x^2}$.

If $y=0$ then the second coordinate of $f_x\times f_y$ is $-\im(\omega_1\overline{\omega_3})=-(1+x^2)(1-x^2+x^4)$, which is never equal to $0$.  

If $y=\pm\sqrt{3x^2}$ then the second coordinate of $f_x\times f_y$ is $-(1+4x^2)(1-4x^2+16x^4)$, which is also never equal to $0$.  Thus, $f$ is regular.

As for $f$ being a proper map, since the first coordinate of $f$ is $-y$, we just need to consider what happens when $y$ is bounded and $x\rightarrow\infty$.  In this case, the third coordinate of $f$ blows up, and so $f$ is proper.
\end{proof}
\begin{figure}[h]
	\centerline{ 
		\includegraphics[width=2.5in]{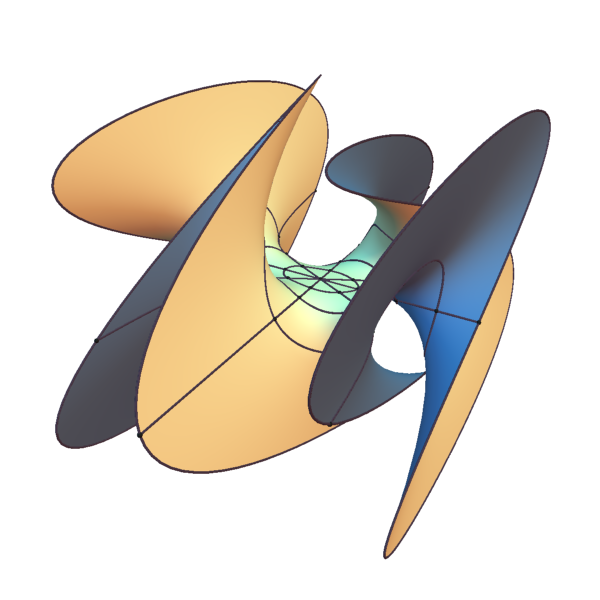}
		\includegraphics[width=2.5in]{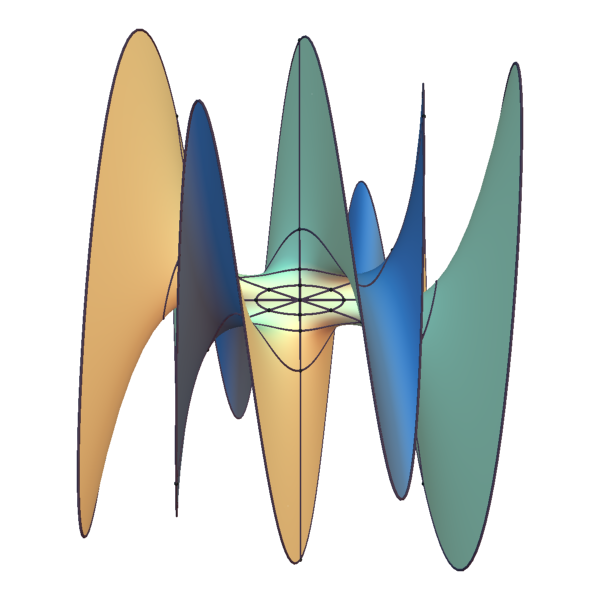}
	}
	\caption{Two views of an end of type $(2,5,8)$}
	\label{figure:258}
\end{figure}

\subsection{An embedded end of type \texorpdfstring{$(3,4,6)$}{(3,4,6)}}

The other complicated embedded end we have found is  of type $(3,4,6)$, 
given by
\begin{example}\label{ex:346}

\begin{align*}
\omega_1&= i  z\, dz
\\
\omega_2&= \left( z^2+a \right) \, dz
\\
\omega_3 &= \left(  i z^4+i\right)\, dz
\end{align*}

and shown for $a=\frac12$ in two views in Figure \ref{figure:346}.
\begin{figure}[h]
	\centerline{ 
		\includegraphics[width=2.5in]{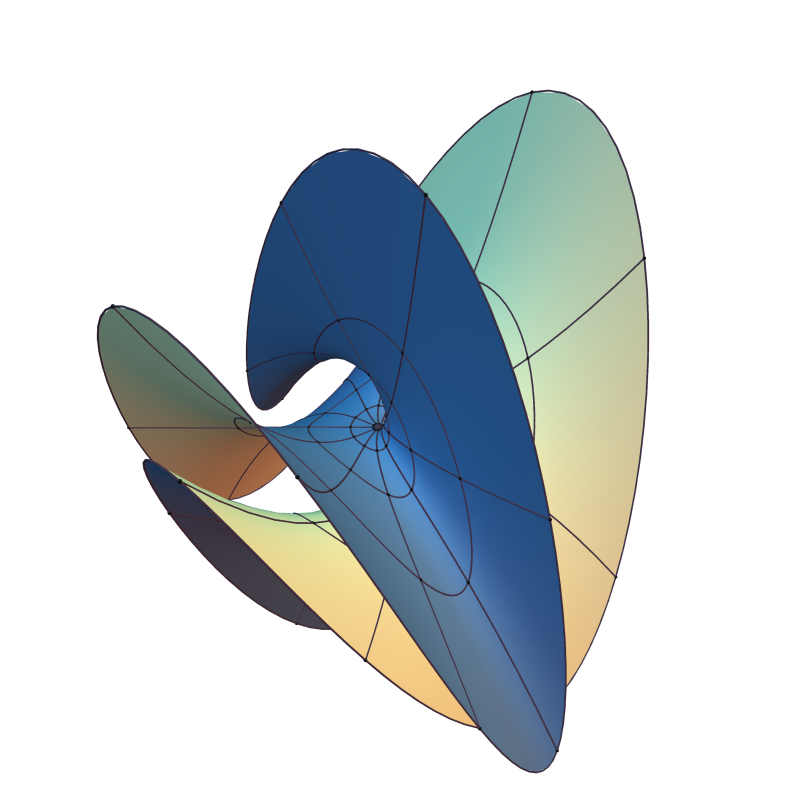}
		\includegraphics[width=2.5in]{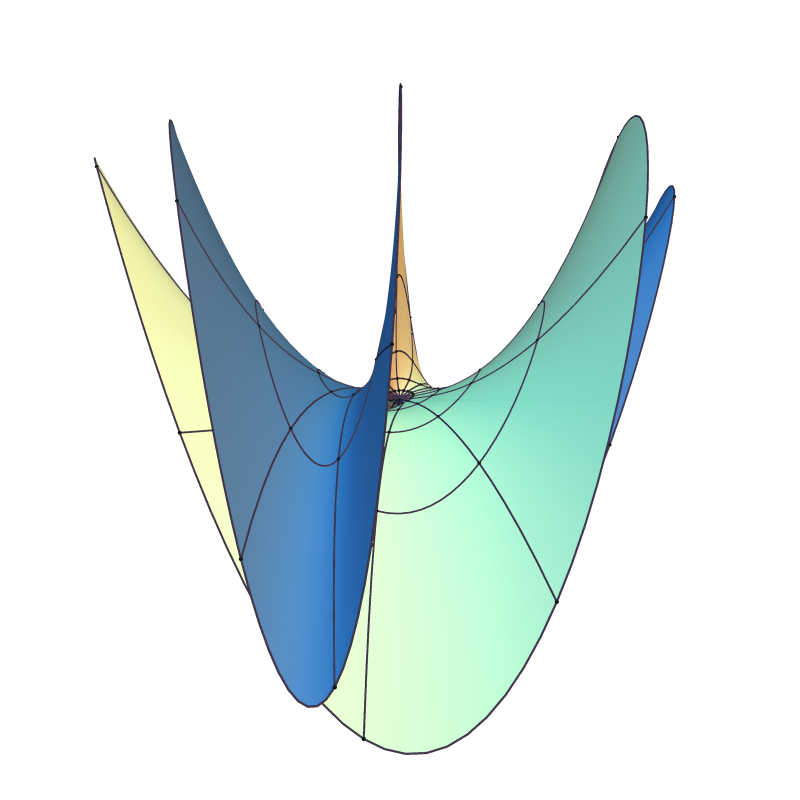}
	}
	\caption{Two views of an end of type $(3,4,6)$}
	\label{figure:346}
\end{figure}

\end{example}

\begin{proposition}
The surface given in Example \ref{ex:346}  is complete and properly embedded for $a=\frac12$.
\end{proposition}

\begin{proof}
The proof is somewhat tedious but rather typical for harmonic surfaces. We first evaluate with $z=x+i y$

\begin{align*}
f(z) &= \int^z (-\omega_1, 3\omega_2, -5\omega_3)\, dz
\\
& = \left(x y,x \left(3 a+x^2-3 y^2\right),y \left(5 x^4-10 x^2
   y^2+y^4+5\right)\right)  
 \end{align*}   
 where we have rescaled the coordinates to avoid irrelevant coefficients.  
 
 For embeddedness we have to show that $f(x_1,y_1)  = f(x_2,y_2)$ implies that $x_1=x_2$ and $y_1=y_2$.
 Let's first assume that we have $x_1=x_2$ but $y_1\ne y_2$. Then
the first coordinate of 
$ f(x,y_1)-f(x,y_2)$ become  $x(y_1-y_2)$ so that $x=0$. This makes the second coordinate of $f(0,y_1)-f(0,y_2)$ equal to 0 as well, but the third becomes
   \[
   \frac{y_1^4}{2}+\frac{y_2^4}{2}+\frac{1}{4}
   \left(y_1^2-y_2^2\right)^2+\frac{1}{4}
   (y_1+y_2)^4+5\ge5>0\ .
   \]
   
   From now on we assume that $x_1\ne x_2$. 
   The next special case is $x_2=0$. Then the first coordinate of $f(x_1,y_1)  = f(0,y_2)$ is $x_1y_1$. As $x_1\ne0$ this implies $y_1=0$, and we get
   \[
   f(x_1,0)-f(0,y_2) = \left(0,x_1(x_1^2+3a), -y_2(y_2^2+5)\right)
   \]
   which is never 0 for positive $a$. 

   Now we turn to the general case, assuming that $x_1\ne x_2$ and $x_1,x_2 \ne 0$. The first coordinate of $f(x_1,y_1)  = f(x_2,y_2)$ is $x_1y_1-x_2y_2$. As $x_2\ne0$, 
   we can solve $y_2= x_1 y_1/x_2$. If $y_1=0$ we get $y_2=0$, and 
   \[
   f(x_1,0)-f(x_2,0)=\left(0,(x_1-x_2)(x_1^2+x_1x_2+x_2^2+3a),0\right)
   \]
   which is never 0 for positive $a$. Hence we can assume that $y_1,y_2\ne0$. Eliminating $y_2$ from $f(x_1,y_1)-f(x_2,y_2)=0$ leaves us with two equations
   
   \[
   a x_2+x_1^2 x_2+x_1 x_2^2+3
   x_1 y_1^2+x_2^3=0
   \]
   and
   
  \[ x_1^4 y_1^4-5 x_1^3 x_2^5+x_1^3
   x_2 y_1^4-5 x_1^2 x_2^6-10
   x_1^2 x_2^4 y_1^2+x_1^2 x_2^2
   y_1^4-5 x_1 x_2^7+x_1 x_2^3
   y_1^4+x_2^4 y_1^4+5 x_2^4=0
   \]
   
   As the first equation contains only $y_1^2$ and the second only even powers of $y_1$, we can eliminate $y_1$ to obtain a single equation  between $x_1$ and $x_2$.
   Using a computer algebra system, it is possible to verify that its only real solution is given by $x_1=x_2=0$ for a range of parameters $a$ that includes $a=\frac12$.
   This shows embeddedness. To show that the surface is regular, look at $f_x\times f_y=0$. Its third coordinate is equal to $-3x(a+x^2+y^2)$, so that $x=0$. In this case
   \[
   f_x\times f_y = \left(15 \left(y^4+1\right) \left(a-y^2\right),-5 y
   \left(y^4+1\right),0\right)
   \]
   which is never 0.
   
   Finally we verify properness. Let $(x,y)\to\infty$, hence $x\to \infty$ or $y\to\infty$. By the first coordinate $x y$ of $f(x,y)$,  $f(x,y)$ can only remain bounded if with $x\to\infty$ we have $y\to 0$, and with $y\to\infty$ we have $x\to 0$. 
   
   Consider first the case that $x\to\infty$ and thus $y\to 0$. Then the second coordinate $x(x^2-3y^2+3a)\to \infty$, and $f$ is proper.
   In case that $y\to\infty$ and thus $x\to 0$, the third coordinate $y(y^4-10x^2y^2+5x^4+5)\to \infty$, and we are proper as well.

   
   \end{proof}

\subsection{Ends of type \texorpdfstring{$(m,n,m+n-1)$}{(m,n,m+n-1}}

The (3,4,6) end actually belongs to a family of immersed ends. The surfaces below have very simple parametrizations in the punctured plane. The type at $0$ is $(m,n,m+n-1)$, while the type at $\infty$ is 
$(2-m,2-n,3-m-n)$ which in general creates a singular surface at $\infty$. 
Due to the twisted appearance of the end it appears to be very hard to incorporate them into complete surfaces.
Moreover, while all numerical evidence points towards it, we do not have an embeddedness proof except for special cases.

\begin{conjecture}
Let $m$, $n$ and $k$  be positive integers such that $(m-1)/k$, $(n-1)/k$, and $(m+n-2)/k$ are relatively prime. 
Assume that $(m-1)/k$ and $(n-1)/k$ have different parity. Then, for sufficiently small disks around $0$, the end
defined by
\begin{align*}
\omega_1&= \frac{i}{z^m}\, dz
\\
\omega_2&=\frac{i}{z^n} \, dz
\\
\omega_3 &= \frac{1}{z^{m+n-1}}\, dz
\end{align*}
is a complete, properly embedded end of type $(m,n,m+n-1)$.
\end{conjecture}

\begin{figure}[h]
	\centerline{ 
		\includegraphics[width=2.5in]{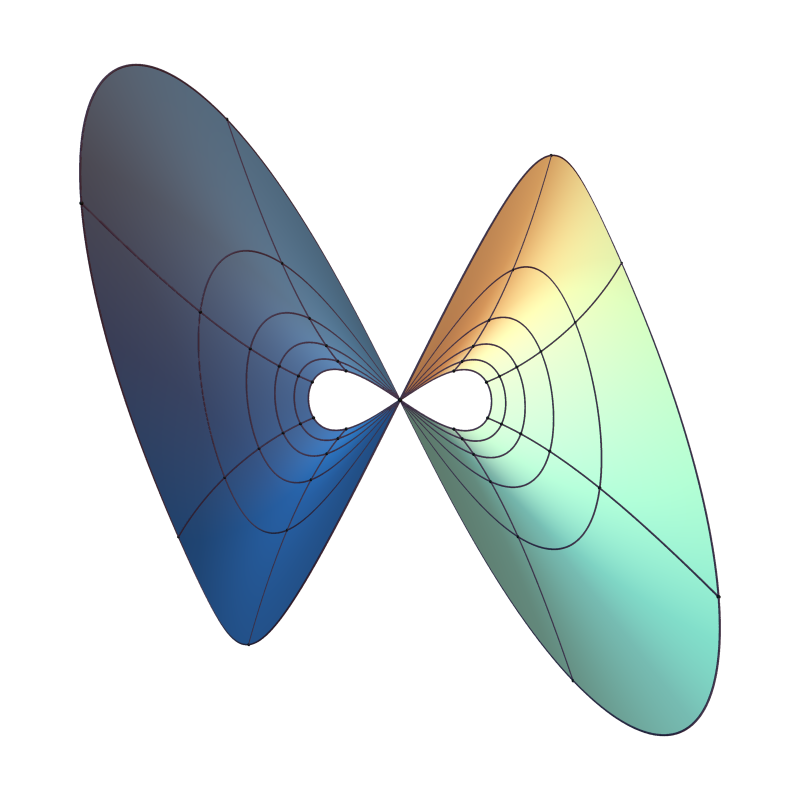}
		\includegraphics[width=2.5in]{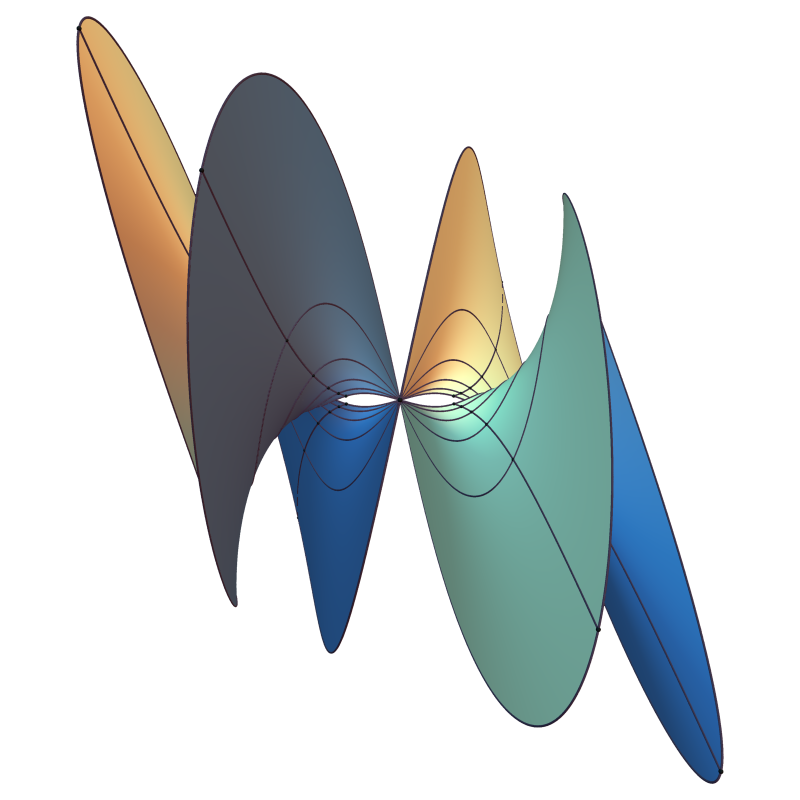}
	}
	\caption{Embedded ends of type $(2,3,4)$ and $(2,5,6)$}
	\label{figure:mn-1}
\end{figure}

\begin{conjecture}
Let $m$, $n$ and $k$  be positive integers such that $(m-1)/k$, $(n-1)/k$, and $(m+n-2)/k$ are relatively prime. 
Assume that $(m-1)/k$ and $(n-1)/k$ have the same parity. Then, for sufficiently small disks around $0$, the end
defined by
\begin{align*}
\omega_1&= \frac{i}{z^m}\, dz
\\
\omega_2&=\frac{1}{z^n} \, dz
\\
\omega_3 &= \frac{i}{z^{m+n-1}}\, dz
\end{align*}
is a complete, properly embedded end of type $(m,n,m+n-1)$.
\end{conjecture}

\begin{figure}[h]
	\centerline{ 
		\includegraphics[width=2.5in]{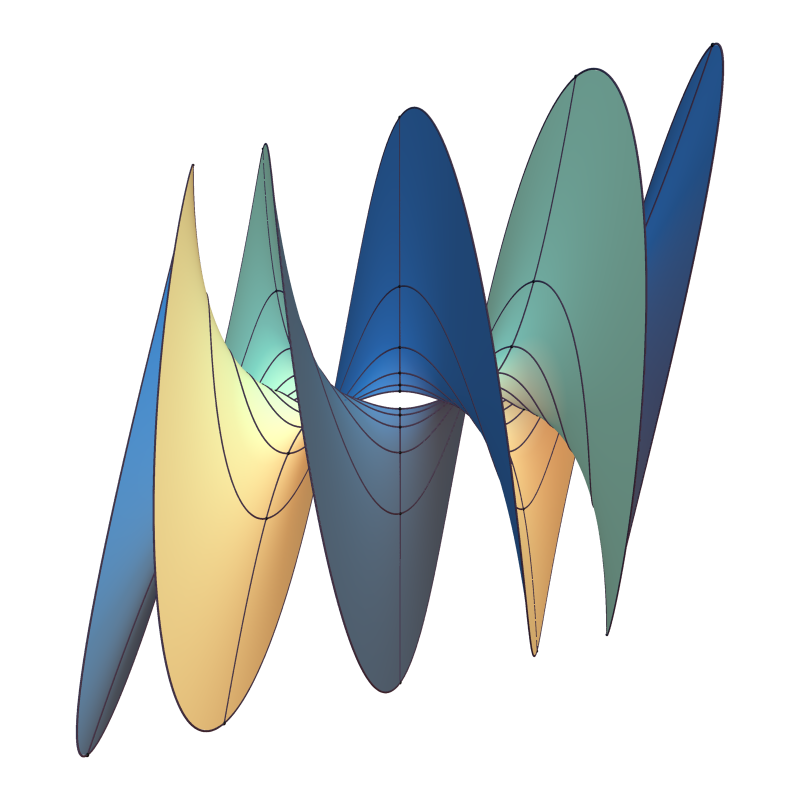}
		\includegraphics[width=2.5in]{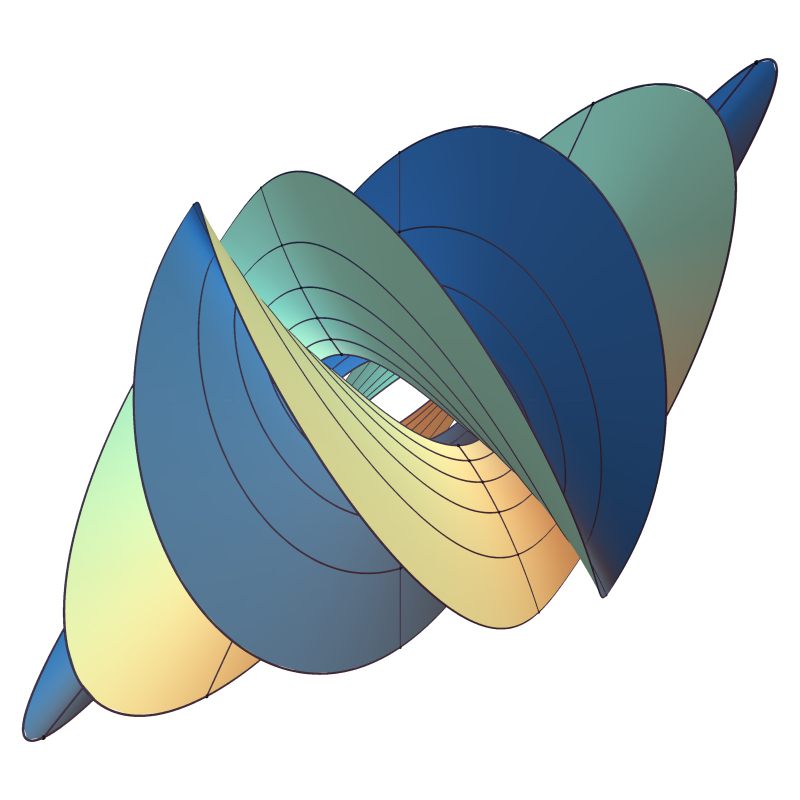}
	}
	\caption{Embedded ends of type $(2,6,7)$ and $(4,6,9)$}
	\label{figure:mn-2}
\end{figure}

\section{Small total curvature}
\label{sec:ftc}
In this section, we give a complete classification of properly embedded harmonic surfaces with total curvature equal to $-2\pi$ and $-4\pi$.  In addition, we classify embedded tori of total curvature $-6\pi$.

Suppose 
\[
f(z)=\re\int^z\left(\omega_1,\omega_2,\omega_3\right)
\]
defines a harmonic map from $X'=X-\{p_1,\ldots,p_m\}$ into $\R^3$, with the genus of $X$ equal to $g$.  By \ref{thm:gaussbonnet}, 

\[
\int_{X'}KdA=-2\pi\left(2g-2+\sum_{j=1}^m n^j\right)
\]
where $n_j$ is the order of the end $p_j$.  If we fix the total curvature and genus then this gives a restriction on the number and order of the ends.

\subsection{Properly embedded harmonic surfaces of total curvature \texorpdfstring{$-2\pi$}{-2pi}}
\label{subsection:-2pi}
There are no minimal surfaces with total curvature $-2\pi$, while two of the simplest harmonic surfaces have total curvature $-2\pi$.

\begin{theorem}
There are two families of properly embedded harmonic surfaces with total curvature $-2\pi$, both genus zero:
\begin{enumerate}
\item
One end of type (0,0,1) and one end of type (1,2,2).
\item
One end of type (2,2,3).
\end{enumerate}
\end{theorem}

\begin{proof}
Suppose the total Gauss curvature equals $-2\pi$.  Then 
\[
3-2g=\sum_{j=1}^m n^j
\]

If the genus of $X$ is $g=0$ then
\[
3=\sum_{j=1}^m n^j
\]
and there are three possible types of surfaces:
\begin{enumerate}
\item
Three ends, each of order one.
\item
Two ends, one of order one and the other of order two.
\item
One end of order three.
\end{enumerate}

A surface with three order one ends would violate the maximum principle because it would be contained in a half-space.  

As for a surface with two ends, of order one and two, we can assume the order one end is at $0$ and the order two end is at $\infty$.  The residues of the $\omega_j$ are real, and so after applying an affine transformation, 
\[
\begin{split}
\omega_1&=\left(a_1+i b_1\right)dz\\
\omega_2&=\left(a_2+i b_2\right)dz\\
\omega_3&=\left(\frac{1}{z}+a_3+i b_3\right)dz
\end{split}
\]
We must have $a_1+i b_1$ and $a_2+i b_2$ independent over $\R$.  Otherwise, $f$ fails to be embedded.  

\begin{figure}[h]
	\centerline{ 
		\includegraphics[width=3in]{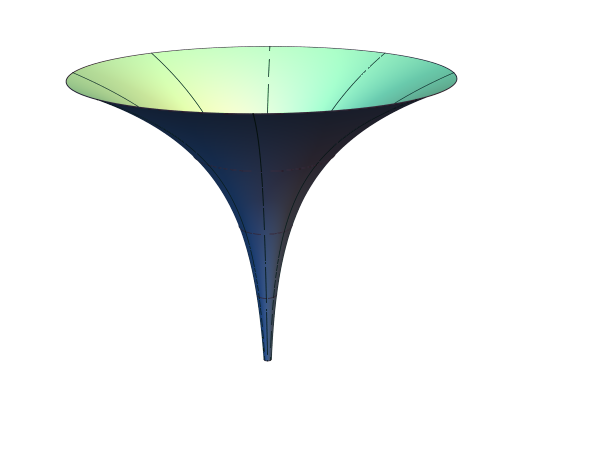}
		\includegraphics[width=3in]{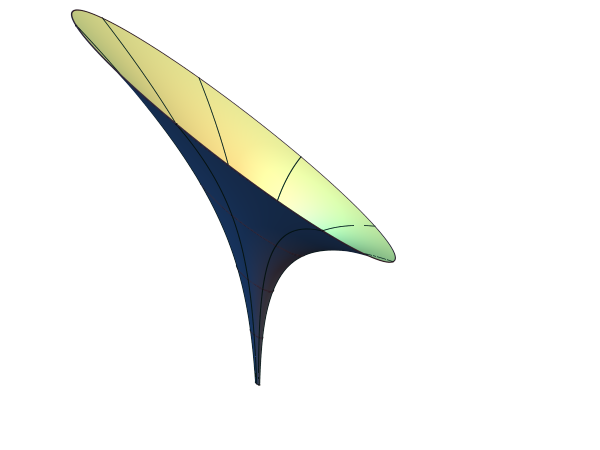}
	}
	\caption{Embedded spheres with one end of type (0,0,1) and one end of type (2,2,1)}
	\label{figure:(0,0,1)+(2,2,1)}
\end{figure}

If we have one end of order three then place the end at $\infty$.  By Lemma \ref{lm.(j,k,k)}, an end of order $(k,3,3)$ can't be embedded.  Hence, our end has order is of the form $(j,k,3)$, with $j\leq k<3$. If $j<2$ then $\omega_1$ has a pole of order $2-j$ at some point $z\in\C$, but then the surface has more than one end.  Thus, the end is of the form $(2,2,3)$:    
\[
\begin{split}
\omega_1&=\left(a_1+ib_1\right)dz\\
\omega_2&=\left(a_2+ib_2\right)dz\\
\omega_3&=\left(a_3+ib_3+(c_3+id_3)z\right)dz\\
\end{split}
\]
If $a_1+ib_1$ and $a_2+ib_2$ independent over $\R$ then $f$ is an embedding.  

\begin{figure}[h]
	\centerline{ 
		\includegraphics[width=3in]{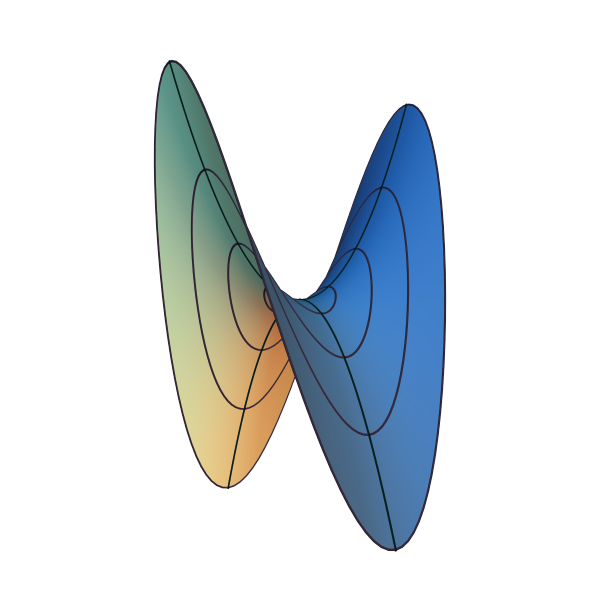}
	}
	\caption{Embedded sphere with one end of type (2,2,3)}
	\label{figure:(2,2,3)}
\end{figure}

If the genus of $X$ is $g=1$ then 
\[
1=\sum_{j=1}^m n^j
\]
and we can have only one end of order one.  This will violate the maximum principle.  
\end{proof}

Both of these examples show the Lopez-Ros Theorem doesn't hold for harmonic surfaces.

\subsection{Properly Embedded Harmonic Surfaces of total curvature\texorpdfstring{$-4\pi$}{-4pi}}
\label{subsection:-4pi}

Osserman showed in \cite{oss1} that the catenoid (two ends of type $(1,2,2)$) and Enneper's surface (one end of type $(3,4,4)$) are the only complete minimal surfaces with total curvature $-4\pi$.  There are many more properly embedded harmonic surfaces with total curvature $-4\pi$.

The examples in this section of surfaces with three or four ends demonstrate that the Hoffman-Meeks Conjecture doesn't hold for harmonic surfaces.

The different examples with two ends show that Schoen's Two-End Theorem also doesn't hold for harmonic surfaces.

\begin{theorem}
There are eleven families of properly embedded harmonic surfaces with total curvature $-4\pi$, all genus zero:
\begin{enumerate}
\item
Four ends of type (0,0,1).
\item
Two ends of type (0,0,1) and one end of type (0,1,2).
\item
Two ends of type (0,0,1) and one end of type (0,2,2).
\item
Two ends of type (0,0,1) and one end of type (1,2,2).
\item
One end of type (0,0,1) and one end of type (1,2,3).
\item
One end of type (0,0,1) and one end of type (2,2,3).
\item
Two ends of type (0,1,2).
\item
One end of type (0,1,2) and one end of type (1,2,2).
\item
Two ends of type (1,2,2).
\item
One end of type (2,2,4).
\item
One end of type (2,3,4).
\end{enumerate}
\end{theorem}

\begin{proof}
Suppose the total Gauss curvature equals $-4\pi$.  Then 
\[
4-2g=\sum_{j=1}^m n^j
\]

If the genus of $X$ is $g=0$ then
\[
4=\sum_{j=1}^m n^j
\]
and there are five possible types of surfaces:
\begin{enumerate}
\item
Four ends, each of order one.
\item
Three ends, two of order one and the other of order two.
\item
Two ends, one of order one and the other of order three.
\item
Two ends, both of order two.
\item
One end of order four.
\end{enumerate}

\subsubsection{Four ends of order one.}

Take $4$ rays in $\R^3$ emanating from the same vertex.  As long as they don't lie in a half-plane, we conjecture that there exists a properly embedded surface with order one ends such that, outside of a compact surface, each end lies in a tubular neighborhood of one of the rays.  We can place the ends at the $3rd$ roots of unity and $\infty$.  After rotating the surface, if necessary, we can assume that the end at $\infty$ points in the direction of the positive $z-axis$ and that the end at $1$ lies in the half-plane $y=0$, $x>0$.  

\begin{proposition}
The surface given on $\C-\{1,e^{2\pi i/3},e^{4\pi i /3}\}$ by

\[
\begin{split}
\omega_1&=\left(\frac{2}{z-1}-\frac{1}{\left(z-e^{2\pi i/3}\right)}-\frac{1}{\left(z-e^{4\pi i/3}\right)}\right)dz=\frac{3(z+1)}{z^3-1}dz\\
\omega_2&=\left(\frac{\sqrt{3}}{z-e^{2\pi i/3}}-\frac{\sqrt{3}}{z-e^{4\pi i/3}}\right)dz=\frac{3i(z-1)}{z^3-1}dz\\
\omega_3&=\left(\frac{1}{z-1}+\frac{1}{z-e^{2\pi i/3}}+\frac{1}{z-e^{4\pi i/3}}\right)dz=\frac{3z^2}{z^3-1}dz\\
\end{split}
\]

is a regular proper embedding.
\end{proposition}

\begin{proof}
If $\alpha(z)=\overline{z}$ and $\sigma(z)=e^{2\pi i/3}$ then 
\[
\alpha^*(\omega_1,\omega_2,\omega_3)=(\overline{\omega_1},-\overline{\omega_2},\overline{\omega_3})
\]
and
\[
\sigma^*(\omega_1+i\omega_2,\omega_3)=\left(e^{2\pi i/3}(\omega_1+i\omega_2),\omega_3\right)
\]
and so the surface is invariant under reflection in the $(x,z)$-plane and a rotation about the $z$-axis by $e^{2\pi i/3}$.  In fact, it is invariant under reflections in planes through the origin we normal vectors $\displaystyle e^{i\pi/6}$ and $\displaystyle e^{-i\pi/6}$.  Using the fact that 
\[
f_x\times f_y=\im\left(\omega_2\overline{\omega_3},-\omega_1\overline{\omega_3},\omega_1\overline{\omega_2}\right)
\]

we can show the surface is regular and embedded by focusing on the region 
\[
X=\{z| 0\leq\arg(z)\leq\pi/3\}
\]
If $z=re^{it}\in X$ with $t>0$ then 
\[
\im(\omega_1\overline{\omega_3})=\frac{9}{|z^3-1|^2}\im\left((z+1)\overline{z}^2\right)-r^2\sin{t}(r+2\cos{t})<0
\]

and so $f(X)$ is a graph over the $xz$-plane.  If $t=0$ then
\[
\im(\omega_1\overline{\omega_2})=-\frac{9(z^2-1)}{(z^3-1)^2}\neq 0
\]
and so $f$ is regular on $\overline{X}$.  Hence, $f$ is regular and embedded.

To see that $f$ is proper, notice that the third coordinate function 
\[
f_3(z)=\log(|z^3-1|)\rightarrow\pm\infty
\]
as $z\rightarrow (1,e^{2\pi i /3}, e^{4\pi i /3},\infty)$.
\end{proof}

\begin{figure}[h]
	\centerline{ 
		\includegraphics[width=2.8in]{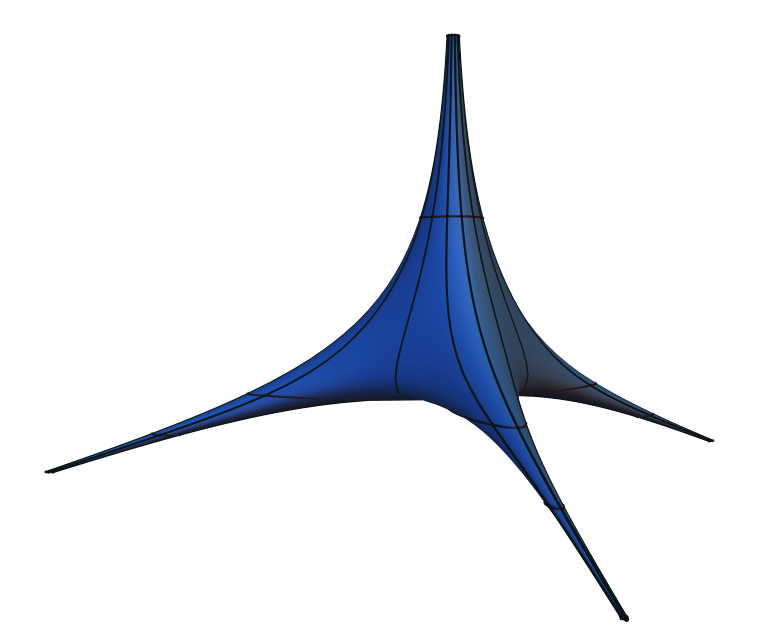}
		\includegraphics[width=2.8in]{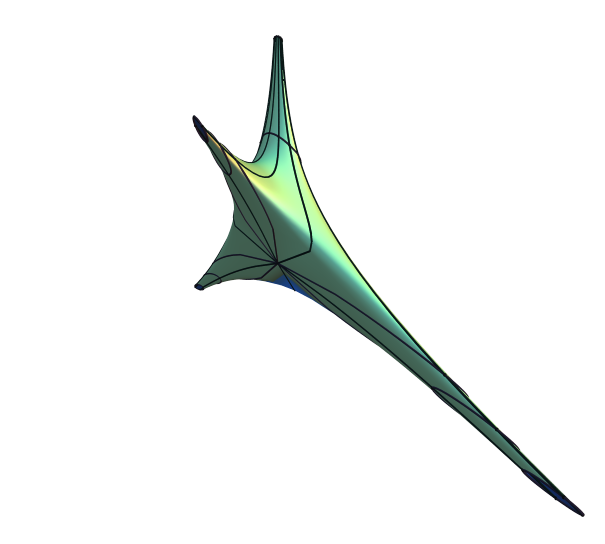}
	}
	\caption{Two embedded spheres with four (0,0,1) ends}
	\label{figure:4(0,0,1)}
\end{figure}

\subsubsection{Three ends, two of order one and the other of order two.}

One can construct properly embedded surfaces with two order one ends and one order two end.  By Theorem \ref{thm:order2}, the order two end can be of type (0,1,2), (0,2,2), or (1,2,2).

If the order two end is of type (0,1,2), we can assume that end opens in the positive z direction around the x-axis (it's actually of the form (2,0,1)).  There is a lot of flexibility for the order one ends.  One needs to point in the positive y direction, and the other needs to point in the negative y direction.  The one forms will be of the form 

\[
\begin{split}
\omega_1&=\left(\frac{a_1}{z-1}+\frac{b_1}{z+1}+c_1\right)dz\\
\omega_2&=\left(\frac{a_2}{z-1}-\frac{a_2}{z+1}\right)dz\\
\omega_3&=\left(\frac{a_3}{z-1}+\frac{b_3}{z+1}\right)dz\\
\end{split}
\]

The surface won't be regular if $c_1\in\R$. 

\begin{figure}[h]
	\centerline{ 
	\includegraphics[width=2.5in]{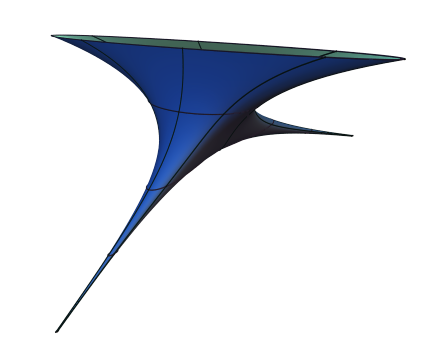}
	}
	\caption{Embedded sphere with  two $(0,0,1)$ ends and one (2,0,1) end}
	\label{figure:(0,0,1)x2+(0,1,2)}
\end{figure}

\begin{proposition}
The surface given on $X=\C-\{\pm 1\}$ by 
\[
\begin{split}
\omega_1&=idz\\
\omega_2&=\left(\frac{1}{z-1}-\frac{1}{z+1}\right)dz\\
\omega_3&=\left(\frac{1}{z-1}+\frac{1}{z+1}\right)dz\\
\end{split}
\]
is a regular proper embedding.
\end{proposition}

\begin{proof}
The surface normal vector is 
\[
f_x\times f_y=\frac{1}{(1 - 2 x + x^2 + y^2) (1 + 2 x + x^2 + y^2)}\left(-8y,-2x(x^2+y^2-1),4(x^2-y^2-1)\right)
\]
which is nonzero on $X$.  Hence, the surface is regular.

In order to show $f$ is embedded, we compute
\[
f(x,y)=\left(-y,\log{\left(\frac{(1-x)^2+y^2}{(1+x)^2+y^2}\right)},2\log{\left(4x^2y^2+(x^2-y^2-1)^2\right)}\right)
\]
Suppose that  $f(x_1,y_1)=f(x_2,y_2)$.  The first coordinate implies that $y_2=y_1$.  The second coordinate implies that $x_1=x_2$ or $y^2=x_1x_2-1$.  If $y^2=x_1x_2-1$ then the third coordinate implies that $x_1^2=x_2^2$.  Hence, $f$ is embedded.

To see that $f$ is proper, note that the third coordinate of $f$ blows up as $z\rightarrow\infty$, and the second coordinate of $f$ blows up as $z\rightarrow\pm 1$.  
\end{proof}

If the order two end is of type (0,2,2), we can assume it has vertical normal (this end will actually be of the form (2,2,0)).  Again, there is a lot of flexibility for the order one ends.  One needs to point in the positive z direction, and the other needs to point in the negative z direction.  The simplest form of this surface is a graph over $\C-\{\pm 1\}$ with one-forms

\[
\begin{split}
\omega_1&=dz\\
\omega_2&=idz\\
\omega_3&=\left(\frac{1}{z-1}-\frac{1}{z+1}\right)dz\\
\end{split}
\]


\begin{figure}[h]
	\centerline{ 
		\includegraphics[width=3in]{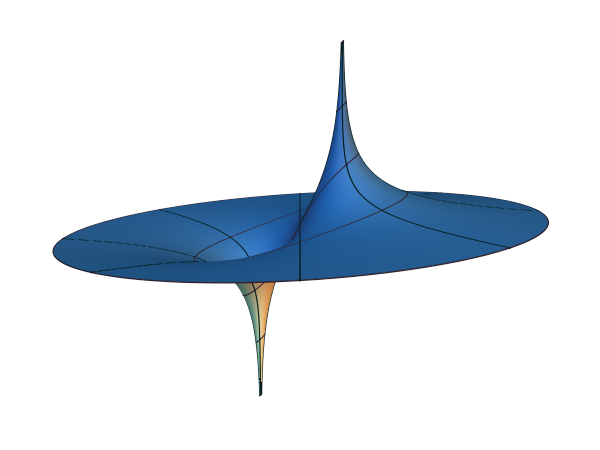}
	}
	\caption{Embedded sphere with two $(0,0,1)$ ends and one $(2,2,0)$ end}
	\label{figure:2(0,0,1)+(2,2,0)}
\end{figure}

If the order two end is of type (1,2,2), we can assume it has vertical normal and opens in the positive z direction.  As long as one of the order one ends points in the negative z direction, the other order one end can point in any direction.  Two versions of this surface - both graphs over $\C-\{\pm 1\}$, have one-forms


\[
\begin{split}
\omega_1&=dz\\
\omega_2&=idz\\
\omega_3&=\left(\frac{1}{z-1}+\frac{1}{z+1}\right)dz\\
\end{split}
\]
and
\[
\begin{split}
\omega_1&=dz\\
\omega_2&=idz\\
\omega_3&=\left(\frac{1}{z-1}-\frac{1}{2(z+1)}\right)dz\\
\end{split}
\]

\begin{figure}[h]
	\centerline{ 
		\includegraphics[width=3in]{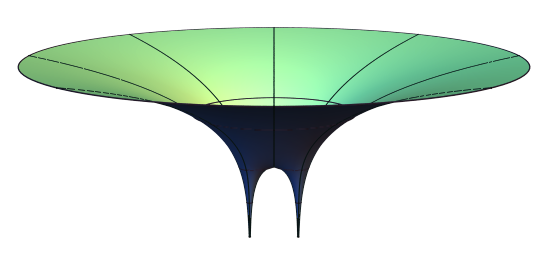}
		\includegraphics[width=3in]{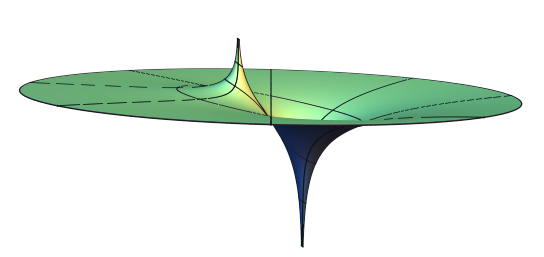}
	}
	\caption{Embedded spheres with two $(0,0,1)$ ends and one $(2,2,1)$ end}
	\label{figure:4(0,0,1)b}
\end{figure}

\subsubsection{Two ends, one of order one and the other of order three}

There are more obstructions to constructing a properly embedded surface with one order one end and one order three end.  Place the order one end at infinity and the order three end at zero.  By Theorem \ref{thm:order3}, the possible order three embedded ends are (0,1,3), (0,2,3), (1,2,3), and (2,2,3).  The (0,1,3) and (0,2,3) cases don't work because that forces the one-form with order zero at $z=0$ to have order two at infinity.  The example with an end of type (1,2,3) is discussed in section \ref{sec:order3}.

One can add an end of order one to the (2,2,3) example given in section \ref{sec:order3} to create a graph over $\C-\{\pm 1\}$  with one forms

\[
\begin{split}
\omega_1&=dz\\
\omega_2&=idz\\
\omega_3&=\left(\frac{1}{z}+z\right)dz\\
\end{split}
\]

\begin{figure}[h]
	\centerline{ 
		\includegraphics[width=3in]{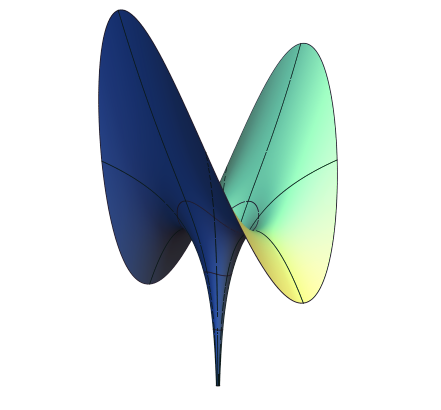}
	}
	\caption{Embedded sphere with (0,0,1) and (2,2,3) ends}
	\label{figure:(0,0,1)+(2,2,3)}
\end{figure}

\subsubsection{Two ends, both of order two.}

In terms of two order two ends, the possible end types are (0,1,2), planar (0,2,2), and catenoid (1,2,2).  Two (1,2,2) ends is well known - the catenoid.  If one of the ends is planar then the maximum principle is violated.  Surfaces with two (0,1,2) ends are discussed in Proposition \ref{thm:(0,1,n)}, section \ref{sec:order3}.  There are examples with ends of type (0,1,2) and (1,2,2).

\begin{proposition}
The surface given on $X=\C^*$ by
\[
\begin{split}
\omega_1&=\left(\frac{i}{z^2}+i\right)dz\\
\omega_2&=dz\\
\omega_3&=\frac{1}{z}dz\\
\end{split}
\]

is a regular proper embedding.
\end{proposition}

\begin{proof}
In order to show $f$ is embedded, we compute
\[
f(r,t)=\left(-\sin{t}\left(\frac{1}{r}+r\right),r\cos{t},\log{r}\right)
\]
which is clearly one-to-one.  The numerator of the surface normal vector is 
\[
f_r\times f_t=\left(\sin{t},-\frac{(1+r^2)\cos{t}}{r^2},r+\frac{\cos(2t)}{r}\right)
\]
and the first two coordinates are never simultaneously equal to zero.

To see that $f$ is proper, note that the third coordinate of $f$ blows up as $z\rightarrow (0,\infty)$.
\end{proof}

\begin{figure}[h]
	\centerline{ 
		\includegraphics[width=3in]{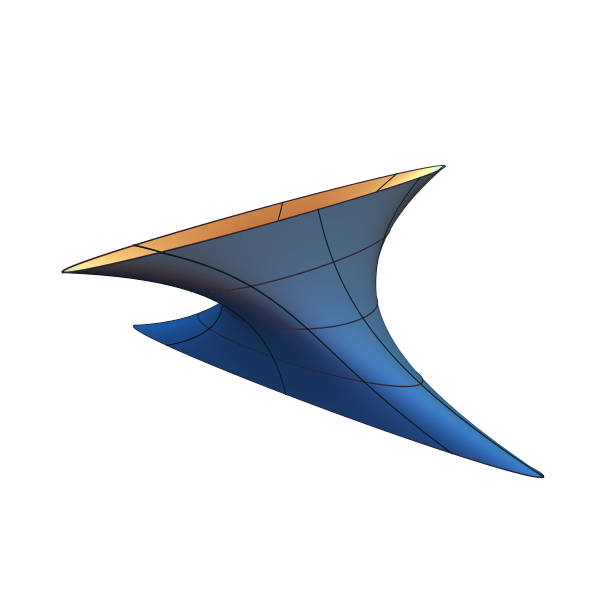}
		\includegraphics[width=3in]{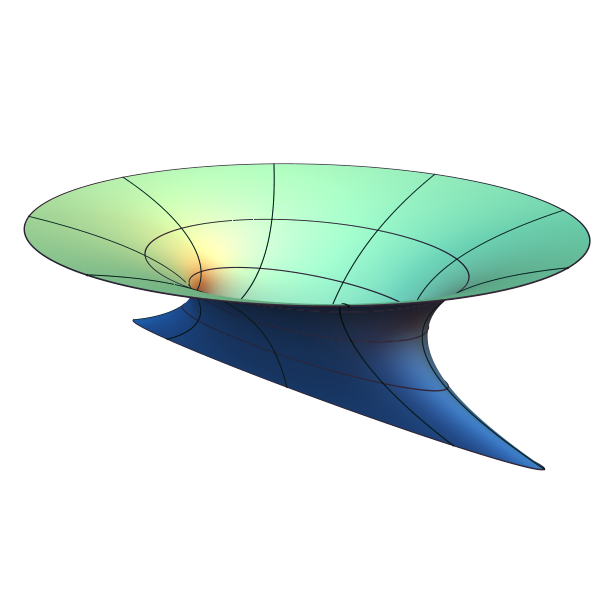}
	}
	\caption{Embedded spheres with two order two ends}
	\label{figure:2(0,1,2),(01,2)+(1,2,2)}
\end{figure}

\subsubsection{One end of order four.}

If a surface has a single order four end then there are several possibilities.   Ends of type (0,1,4), (0,2,4), (0,3,4), (1,2,4), and (1,3,4) can't exist because all would force another end somewhere else or would violate the maximum principle.  Ends of type (3,3,4) and  $(n,4,4)$ aren't embedded, by Lemma \ref{lm.(j,k,k)}.  The remaining types of (2,2,4) and (2,3,4) can form properly embedded surfaces and are discussed in sections \ref{sec:order3} and \ref{sec:(2,3,n)}.

If the genus $g=1$ then $\displaystyle \sum_{k=1}^m\max(n_k)=2$, and there must be two ends of order one or one end of order two.  Both possibilities violate the maximum principle.  Therefore, there are no  tori with total Gauss curvature $-4\pi$.  There can't be any surfaces with $g>1$ either, because then $\displaystyle \sum_{k=1}^m\max(n_k)\leq 0$.
\end{proof}

\subsection{Embedded tori of total curvature \texorpdfstring{$-6\pi$}{-6\pi}}

A classification of embedded harmonic surfaces of total curvature $-6\pi$ would be very tedious. However, 
if one limits one's attention to tori, the classification of ends of low order and the maximum principle together with the Gauss-Bonnet theorem imply:

\begin{theorem}
The only complete, properly embedded tori with total curvature $-6\pi$ can be tori with 
\begin{enumerate}
\item one end of type $(0,0,1)$ and one end of type $(2,2,1)$, or
\item one end of type $(2,2,3)$.
\end{enumerate}
\end{theorem}

Note that these examples are derived from the simply connected examples with total curvature $-2\pi$ by adding a handle to each one.  In a later paper, we will demonstrate the same is true for harmonic tori with total curvature $-8\pi$ - they are created by adding a handle to each example with total curvature $-4\pi$.

We will give examples of such surfaces in sections \ref{sec:surf(0,0,1)(2,2,1)} and \ref{sec:surf223}.  The torus with ends of type $(0,0,1)$ and $(2,2,1)$ is another example demonstrating that Schoen's Two-End Theorem doesn't hold for harmonic surfaces.

\section{Embedded examples of finite topology}
\label{sec:embedded}

In this section, we give examples of complete, properly embedded  harmonic surfaces showing off the types of embedded ends we have found.
The method of construction is very similar to the one described in Karcher's Tokyo notes \cite{karch3} using the Weierstrass representation for minimal surfaces: Assuming symmetries 
and other  geometric features, one obtains candidate Weierstrass data for a minimal surface. Instead of going through the often difficult (or impossible!) process of closing the periods by adjusting the parameters to obtain a minimal surface, we force close the periods by adding holomorphic 1-forms to the coordinate 1-forms coming from the Weierstrass data. This ensures that we keep the same asymptotic behavior, but forsakes of course the minimality. This simplification is counterbalanced by the added difficulties:

\begin{enumerate}
\item embedded ends of harmonic surfaces can be much more complicated;
\item due to the lack of conformality, the regularity of the parametrization is harder to check;
\item embeddedness is not preserved under deformations due to a lack of a strong maximum principle.
\end{enumerate}

\subsection{Tori with one end of type \texorpdfstring{$(0,0,1)$}{(0,0,1)} and one end of type \texorpdfstring{$(2,2,1)$}{(2,2,1)}}  
\label{sec:surf(0,0,1)(2,2,1)}

It is possible to add a handle to the genus 0 surface with one end of type $(0,0,1)$ and one end of type $(2,2,1)$.

\begin{figure}[h]
	\centerline{ 
		\includegraphics[width=2.5in]{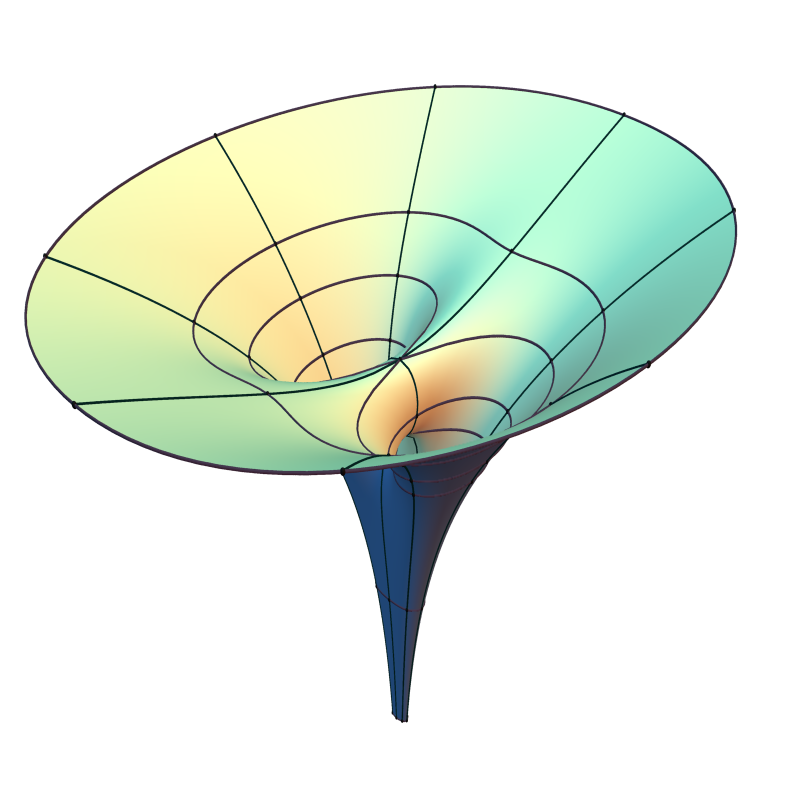}
	}
	\caption{Embedded Tori with ends $(0,0,1)$ and $(2,2,1)$}
	\label{figure:g1(1,2,2)+(0,0,1)}
\end{figure}

Using the torus given by 
\[
X=\{(z,w)\in\C^2 | w^2=z(z-2)(z-1/2)\}
\]
the 1-forms 
\[
\begin{split}
\left(\omega_1,\omega_2,\omega_3\right)&=\left(\frac{z}{w},\frac{iz}{w},\frac{1}{z}\right)dz\\
\end{split}
\]

have poles of order $(0,0,1)$ at $(z,w)=(0,0)$ and $(2,2,1)$ at $(z,w)=(\infty,\infty)$ and are holomorphic elsewhere.  The tentative parametrization
\[
f(z)=\re\int^z\left(\omega_1,\omega_2,\omega_3\right)
\]
may have some periods that need to be closed.  Without adjusting $a$, we can get a single valued parametrization by adding the holomorphic forms $\displaystyle\frac{\lambda_1}{w}$ and $\displaystyle\frac{i\lambda_2}{w}$ to $\omega_1$ and $\omega_2$, respectively.  The adjusted 1-forms are

\begin{equation}
\omega_1=\frac{z+\lambda_1}{w}dz,\hspace{.2in} \omega_2=i\frac{z+\lambda_2}{w},\hspace{.2in} \omega_3=\frac{1}{z}dz,\hspace{.2in} \Omega=(\omega_1,\omega_2,\omega_3).
\end{equation}

\begin{proposition}
There exist $\lambda_1,\lambda_2\in\R$ such that 
\begin{equation}
\displaystyle f(z)=\re\int^z\Omega
\label{eq51}
\end{equation}
is a proper harmonic embedding of $X$ into $\R^3$.
\end{proposition}

The following two lemmas provide the proof of this proposition.
\begin{lemma}
There exist $\lambda_1\in(-1/2,0)$ and $\displaystyle\lambda_2\in\left(-\frac{5}{4},-1\right)$ such that the periods produced in equation \ref{eq51} are all zero.
\label{lem51a}
\end{lemma}

\begin{proof}
Let $\gamma_1$ be the circle with radius $1/2$ and center $(1/4,0)$ and $\gamma_2$ be the circle with radius $1$ and center $5/4$.  Denote the lifts of $\gamma_1$ and $\gamma_2$ to $\overline{X}$ with the same notation.  They form a basis of $H_1(X)$.

Observe that $\omega_3$ has no real periods on $X$.  In order to compute the periods of $\omega_1$ and $\omega_2$ on $\gamma_1$ and $\gamma_2$, collapse $\gamma_1$ to the interval $(0,1/2)$ and $\gamma_2$ to the interval $(1/2,2)$.  Then $\displaystyle\re\int_{\gamma_1}\omega_2=\re\int_{\gamma_2}\omega_1=0$ for all $\gamma_1,\gamma_2\in\R$.  

If $\lambda_1=0$ then 
\[
\re\int_{\gamma_1}\omega_1=2\re\int_0^{1/2}\frac{z}{\sqrt{z}\sqrt{z-2}\sqrt{z-1/2}}dz=-2\int_0^{1/2}\frac{z}{\sqrt{z}\sqrt{2-z}\sqrt{1/2-z}}dz\leq 0
\]
If $\lambda_1=-1/2$ then
\[
\re\int_{\gamma_1}\omega_1=2\re\int_0^{1/2}\frac{z-1/2}{\sqrt{z}\sqrt{z-2}\sqrt{z-1/2}}dz=2\int_0^{1/2}\frac{1/2-z}{\sqrt{z}\sqrt{2-z}\sqrt{1/2-z}}dz\geq 0
\]
By the intermediate value theorem there exists $\lambda_1\in(-1/2,0)$ such that $\displaystyle\re\int_{\gamma_1}\omega_1=0$.

If $\lambda_2=-1$ then
\[
\begin{split}
\re\int_{\gamma_2}\omega_2&=2\re\int_{1/2}^{2}\frac{i(z-1)}{\sqrt{z}\sqrt{z-2}\sqrt{z-1/2}}dz\\
&=2\re\int_{1/2}^{2}\frac{z-1}{\sqrt{z}\sqrt{2-z}\sqrt{z-1/2}}dz\\
&=2\re\int_{1/2}^{1}\frac{z-1}{\sqrt{z}\sqrt{2-z}\sqrt{z-1/2}}dz+2\re\int_1^{2}\frac{z-1}{\sqrt{z}\sqrt{2-z}\sqrt{z-1/2}}dz\\
&=2\re\int_{1/2}^{1}\frac{z-1}{\sqrt{z}\sqrt{2-z}\sqrt{z-1/2}}dz-2\re\int_{1/2}^{1}\frac{z-1}{z^{3/2}\sqrt{2-z}\sqrt{z-1/2}}dz\\
&=2\re\int_{1/2}^{1}\frac{(z-1)^2}{z^{3/2}\sqrt{2-z}\sqrt{z-1/2}}dz\\
&\geq 0\\
\end{split}
\]

If $\lambda_2=-\frac{5}{4}$ then
\[
\begin{split}
\re\int_{\gamma_2}\omega_2&=2\re\int_{1/2}^{2}\frac{i\left(z-\frac{5}{4}\right)}{\sqrt{z}\sqrt{z-2}\sqrt{z-1/2}}dz\\
&=2\re\int_{1/2}^{2}\frac{\left(z-\frac{5}{4}\right)}{\sqrt{z}\sqrt{2-z}\sqrt{z-1/2}}dz\\
&=2\re\int_{1/2}^{5/4}\frac{\left(z-\frac{5}{4}\right)}{\sqrt{z}\sqrt{2-z}\sqrt{z-1/2}}dz+2\re\int_{5/4}^{2}\frac{\left(z-\frac{5}{4}\right)}{\sqrt{z}\sqrt{2-z}\sqrt{z-1/2}}dz\\
&=2\re\int_{1/2}^{5/4}\frac{\left(z-\frac{5}{4}\right)}{\sqrt{z}\sqrt{2-z}\sqrt{z-1/2}}dz+2\re\int_{1/2}^{5/4}\frac{\left(1-\frac{5}{4}z\right)}{z^{3/2}\sqrt{2-z}\sqrt{z-1/2}}dz\\
&=2\re\int_{1/2}^{4/5}\frac{z^2-\frac{5}{2}z+1}{z^{3/2}\sqrt{2-z}\sqrt{z-1/2}}dz+2\re\int_{4/5}^{5/4}\frac{z-5/4}{z^{3/2}\sqrt{2-z}\sqrt{z-1/2}}dz\\
&\leq 0\\
\end{split}
\]
By the intermediate value theorem there exists $\lambda_2\in\left(-2,-\frac{5}{4}\right)$ such that $\displaystyle\re\int_{\gamma_2}\omega_2=0$.
\end{proof}

\begin{lemma}
For $\lambda_1$ and $\lambda_2$ as given in Lemma \ref{lem51a}, the map $\displaystyle f(z)=\re\int^z\Omega$ is properly embedded.
\end{lemma}

\begin{figure}[h]
	\centerline{ 
		\includegraphics[width=2.5in]{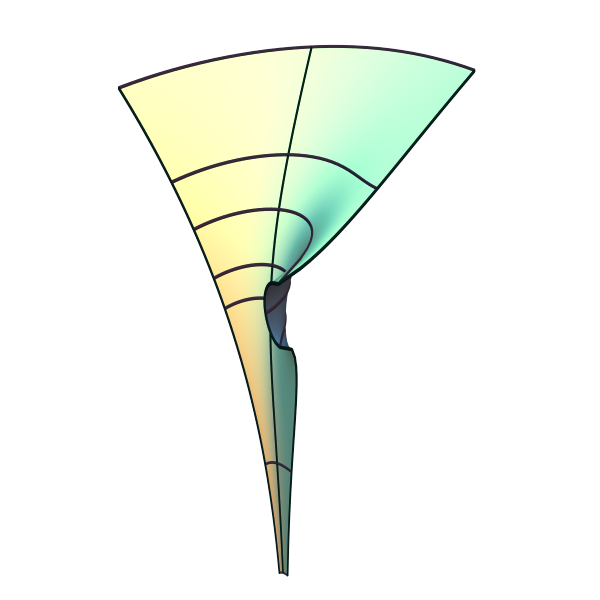}
	}
	\caption{One fourth of an embedded torus with ends of type $(0,0,1)$ and $(2,2,1)$}
	\label{figure:g1(2,2,1)+(0,0,1)b}
\end{figure}

\begin{proof}
We will show that a fundamental piece of the surface over the domain $X_1=\{(w,z)|\im z>0, w>0\}$ is a graph over the $x_2x_3$-plane.  The numerator of the first coordinate of the Gauss map is given by
\[
\re\left(\frac{(z+\lambda_2)\sqrt{z}}{\sqrt{z-2}\sqrt{z-1/2}}\right)
\]
This equals zero when 
\begin{equation}
\left(\frac{(z+\lambda_2)\sqrt{z}}{\sqrt{z-2}\sqrt{z-1/2}}\right)^2
\label{eq5.1}
\end{equation}
is a non-positive real number.  We will show this never happens when $\im z>0$.  Let $z=x+iy$.  The imaginary part of equation \ref{eq5.1} is given by 
\[
\frac{1}{3}y\left(3-\frac{4(2+\lambda_2)^2}{(x-2)^2+y^2}+\frac{(a+2\lambda_2)^2}{(1-2x)^2+4y^2}\right)
\]
which equals zero when $y=0$,
\[
y=y_1=\frac{\sqrt{1+5\lambda_2+\lambda_2^2+(5-2x)x-\sqrt{(1+3\lambda_2+\lambda_2^2+x)(1+7\lambda_2+\lambda_2^2+9x)}}}{\sqrt{2}}
\]
or 
\[
y=y_2=\frac{\sqrt{1+5\lambda_2+\lambda_2^2+(5-2x)x+\sqrt{(1+3\lambda_2+\lambda_2^2+x)(1+7\lambda_2+\lambda_2^2+9x)}}}{\sqrt{2}}
\]
We disregard $y=0$ because we are only concerned about $y>0$.  Simple calculus computations show that $y_1\notin\R$ and $y_2\notin\R$ when $x<-\lambda_2$.  Hence, we examine equation \ref{eq5.1} when $x\geq\lambda_2$ and $y=y_2$.  When $y=y_2$, the real part of equation \ref{eq5.1} is 

\[
\frac{-1+\lambda_2^2+\lambda_2(-5+8x)+x(-5+8x)-\sqrt{(1+3\lambda_2+\lambda_2^2+x)(1+7\lambda_2+\lambda_2^2+x)}}{4x-5}
\]

Again, simple calculus computations show this is positive when $x\geq\lambda_2$.  Hence, if $\im z>0$ then
\[
\left(\frac{(z+\lambda_2)\sqrt{z}}{\sqrt{z-2}\sqrt{z-1/2}}\right)^2
\]
is never a non-positive real number.

Let $\tau(z,w)=\left(\overline{z},-\overline{w}\right)$ and $\sigma(z,w)=\left(\overline{z},\overline{w}\right)$.  Then, $\tau^*\left(\omega_1,\omega_2\omega_3\right)=\left(-\overline{\omega_1},\overline{\omega_2},\overline{\omega_3}\right)$ and $\sigma^*\left(\omega_1,\omega_2,\omega_3\right)=\left(\overline{\omega_1},-\overline{\omega_2},\overline{\omega_3}\right)$.  Thus, $\tau$ induces a reflection in a plane parallel to the $x_2x_3$-plane and $\sigma$ induces a reflection in a plane parallel to the $x_1x_3$-plane, and $f(X_1)$ is bounded by these planes.  On the $z$ plane, the fixed point sets of $\tau$ and $\sigma$ are the real intervals $(-\infty,0)\cup[1/2,2]$ and $(0,1/2]\cup[2,\infty)$, respectively.  The fixed-point sets of these reflections are the same as the images of the fixed point sets of $\tau$ and $\sigma$.

Lastly, we need to show that the boundary of $f(X_1)$ is embedded.  The boundary of $f(X_1)$ consists of two curves: $f((-\infty,0))$ and $f((0,\infty))$.  As $f_3(z)=\log{|z|}$, $f$ is one-to-one on each of these curves.  The pieces $f((-\infty,0))$ and $f([1/2,2])$ lie in the $x_2x_3$-plane, whereas $f((0,1/2))$ and $f((2,\infty))$ do not.  Thus, the only chance for an intersection along the boundary would be at $x$ and $-x$, with $1/2\leq x\leq 2$.  The first coordinate of the Gauss map will be zero at $x$ and $-x$, the second coordinate of the Gauss map is negative at $x$ and positive at $-x$.  Consider the curve $\gamma_x(t)=xe^{\pi i t}$ for $0\leq t\leq1$.  Note that $f\circ\gamma_x$ is a level curve with height $\log{x}$.  The first coordinate of the Gauss map is positive along $\gamma_x$ for $0<t<1$, and so $f_2$ is one-to-one along $\gamma_x$.  Hence, $f(-x)\neq f(x)$, and the boundary of $f(X_1)$ is embedded.  Therefore, $f(X)$ is embedded.

As the third coordinate function is $f_3(z)=\log{|z|}$, $f$ is proper.  
\end{proof}

\subsection{Tori with a single end of type \texorpdfstring{$(2,2,3)$}{(2,2,3)}}
\label{sec:surf223}

The embedded end of type $(2,2,3)$ is a genus 0 surface just by itself.
At  first glance it has a similar appearance as the minimal Enneper end which is of type $(3,4,4)$ (but not embedded). 
Many of the constructions that are known for Enneper's end can be carried out (with less effort) for the end of type $(2,2,3)$.
For instance, there are
higher symmetry cousins of type $(2,2,n)$ and  (now embedded) higher genus harmonic surfaces..

\begin{figure}[h]
	\centerline{ 
		\includegraphics[width=2.5in]{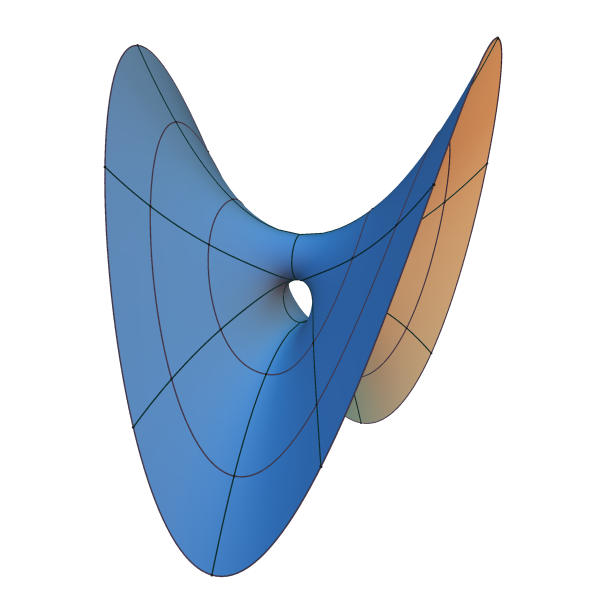}
		\includegraphics[width=2.5in]{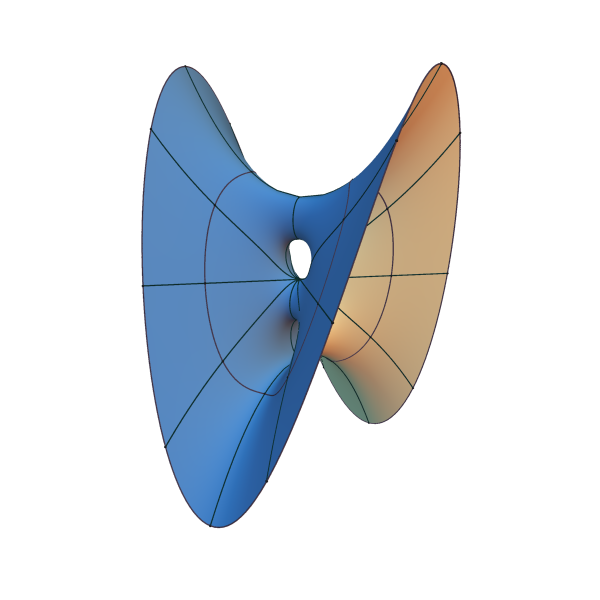}
	}
	\caption{An embedded  end of type $(2,2,3)$ for genus 1 and 2}
	\label{figure:embenneper}
\end{figure}

\begin{theorem}
There exists a complete, properly embedded harmonic  torus with one end of type $(2,2,3)$. 
\end{theorem}

\begin{remark}
In contrast, by work of \cite{tha3,sa1,ww1}, there are complete, non-embedded minimal surfaces of any genus with a single Enneper end.
\end{remark}

The embedded sphere with one end of type $(2,2,3)$ - the hyperbolic paraboloid - has Weierstrass representation given by 
\[
(\omega_1,\omega_2,\omega_3)=(1,i,z)dz
\]
with the end placed at infinity.  We can use this as our model for creating an embedded torus with one end of type (2,2,3).  Using the torus given by 
\[
X=\{(z,w)\in\C^2 | w^2=z(z-1)(z+1)\}
\]
the 1-forms 
\[
\begin{split}
\left(\omega_1,\omega_2,\omega_3\right)&=\left(\frac{1}{\sqrt{z}\sqrt{z-1}\sqrt{z+1}},\frac{i}{\sqrt{z}\sqrt{z-1}\sqrt{z+1}},1\right)dz\\
&=\left(\frac{z}{w},\frac{iz}{w},1\right)dz\\
\end{split}
\]

have poles of order (2,2,3) at $(\infty,\infty)$ and are holomorphic elsewhere.  We close the periods of
\[
f(z)=\re\int^z\left(\omega_1,\omega_2,\omega_3\right)
\]
using the  holomorphically adjusted 1-forms 

\begin{equation}
\omega_1=\frac{z+\lambda_1}{w}dz,\hspace{.2in} \omega_2=i\frac{z+\lambda_2}{w},\hspace{.2in} \omega_3=dz,\hspace{.2in} \Omega=(\omega_1,\omega_2,\omega_3).
\end{equation}

\begin{proposition}
There exist $\lambda_1,\lambda_2\in\R$ such that 
\begin{equation}
\displaystyle f(z)=\re\int^z\Omega
\label{eq2}
\end{equation}
is a proper harmonic embedding of $X$ into $\R^3$.
\end{proposition}

\begin{lemma}
There exist $\lambda_1,\lambda_2\in\R$ such that the periods produced in \ref{eq2} are all zero.
\end{lemma}

\begin{proof}
$\omega_1,\omega_2,$, and $\omega_3$ are all holomorphic 1-forms on $X$ and have poles only at $z=\infty$.  Thus, they have no residues, and we just need to check whether $\displaystyle \re\int_{\gamma}\Omega=0$ for each $\gamma\in H_1(X)$.

Let $\gamma_1$ and $\gamma_2$ be circles centered at $-\frac{1}{2}$ and $\frac{1}{2}$ with radius 1.  Denote the lifts of $\gamma_1$ and $\gamma_2$ to $\overline{X}$ with the same notation.  They form a basis of $H_1(X)$.  

Since $\omega_3=dz$ is exact, $\displaystyle \re\int_{\gamma_i}\omega_3=0$ for $i=1,2$.

In order to compute the periods of $\omega_1$ and $\omega_2$ on $\gamma_1$ and $\gamma_2$, collapse $\gamma_1$ to the interval $(-1,0)$ and $\gamma_2$ to the interval $(0,1)$.  Then $\displaystyle \re\int_{\gamma_1}\omega_2=\re\int_{\gamma_2}\omega_1=0$ for all $\lambda_1,\lambda_2\in\R$.  Hence, we only need to find $\lambda_1,\lambda_2\in\R$ such that $\displaystyle \re\int_{\gamma_1}\omega_1=\re\int_{\gamma_2}\omega_2=0$.  

We use the intermediate value theorem.  If $\lambda_1=0$ then 
\[
\re\int_{\gamma_1}\omega_1=2\int_{-1}^0\frac{z}{\sqrt{z}\sqrt{z-1}\sqrt{z+1}}dz=2\int_{-1}^0\frac{-z}{\sqrt{-z}\sqrt{1-z}\sqrt{z+1}}dz\geq 0.
\]

If $\lambda_1=1$ then 
\[
\re\int_{\gamma_1}\omega_1=2\int_{-1}^0\frac{z+1}{\sqrt{z}\sqrt{z-1}\sqrt{z+1}}dz=2\int_{-1}^0\frac{-(z+1)}{\sqrt{-z}\sqrt{1-z}\sqrt{z+1}}dz\leq 0.
\]

Thus, by the intermediate value theorem, there exists $\lambda_1\in(0,1)$ such that $\displaystyle \re\int_{\gamma_1}\omega_1=0$.  

If $\lambda_2=-\lambda_1$ then

\[
\re\int_{\gamma_2}\omega_2=-\re\int_{\gamma_1}\omega_1=0
\]
\end{proof}

\begin{lemma}
The map given by \ref{eq2} is an immersion.
\end{lemma}

\begin{proof}
In order for $f$ to be an immersion, we need 
\[
f_x\times f_y=\im\left(\omega_2\bar{\omega}_3,-\omega_1\bar{\omega_3},\omega_1\bar{\omega}_2\right)\neq 0
\]

If $w\neq 0$ (that is $z\neq 0$, $z \neq 1$, and $z\neq -1$), then
\[
\im\omega_2\bar{\omega}_3=\im\left(\frac{i(z-\lambda_1)}{w}\right)=\re\left(\frac{z-\lambda_1}{w}\right)=\re\left(\frac{z}{w}\right)-\re\left(\frac{\lambda_1}{w}\right)
\]
and
\[
\im(-\omega_1\bar{\omega}_3)=\im\left(-\frac{z+\lambda_1}{w}\right)=-\im\left(\frac{z}{w}\right)-\im\left(\frac{\lambda_1}{w}\right)
\]
If $\im\omega_2\bar{\omega}_3=\im(-\omega_1\bar{\omega}_3)=0$ then 
\[
\re\left(\frac{z}{w}\right)=\re\left(\frac{\lambda_1}{w}\right)
\]
and 
\[
\im\left(\frac{z}{w}\right)=-\im\left(\frac{\lambda_1}{w}\right)
\]
Hence,
$\displaystyle \frac{z}{w}=\frac{\lambda_1}{\bar{w}}\Rightarrow z|w|^2=\lambda_1w^2\Rightarrow|z||z^2-1|=\lambda(z^2-1)$, but then we must have $z=\pm\lambda_1$ and $\lambda_1>1$.  This  contradicts the fact that $0<\lambda_1<1$.

We also need to check when $w=0$, that is, when $z=0$, $z=1$, or $z=-1$.  In a neighborhood of $w=0$, we can write 
\[
\omega_1=\frac{2(z+\lambda_1)}{3z^2-1}dw, \hspace{.3in} \omega_2=\frac{2(z-\lambda_1)i}{3z^2-1}dw
\]
If $z=0$, $z=1$, or $z=-1$ then
\[
\im\left(\omega_1\bar{\omega}_2\right)=\im\left(\frac{-4(z^2-\lambda_1^2)i}{(3z^2-1)^2}\right)\neq 0
\]
\end{proof}

\begin{lemma}
$f(X)$ is a properly embedded surface.
\end{lemma}

\begin{proof}
The Riemann surface $\overline{X}$ has the automorphisms 

$\displaystyle \alpha:(z,w)\rightarrow (z,-w)$ and $\displaystyle \tau:(z,w)\rightarrow (\overline{z},\overline{w})$, and 
\[
\begin{split}
\alpha^*(\omega_1,\omega_2,\omega_3)&=(-\omega_1,-\omega_2,\omega_3)\\
\tau^*(\omega_1,\omega_2,\omega_3)&=(\overline{\omega}_1,-\overline{\omega}_2,\overline{\omega}_3)\\
\end{split}
\]

Thus, $\tau$ induces a reflection in a plane parallel to the $x_1x_3$-plane and $\alpha\circ\tau$ induces a reflection in a plane parallel to the $x_2x_3$-plane, and so $f(X)$ is invariant under reflection in two orthogonal planes.

The fixed-point sets of these reflections are the same as the images of the fixed point sets of $\tau$ and $\alpha\circ\tau$.  The fixed point set of $\tau$ is the real intervals $[-1,0]\cup[1,\infty)$.  The fixed point set of $\alpha\circ\tau$ is the real intervals $(-\infty,-1]\cup[0,1]$.

$f(X)$ consists of four congruent pieces that are bounded by four symmetry lines.  We will show that $f(X)$ is an embedding by showing that each congruent piece is a graph over a plane.

Let $X_1$ be the sheet of $X$ covering the upper half of the z-plane with $w(2)=\sqrt{6}$.  Then $f(X_1)$ is one of the congruent pieces whose boundary is given by $f(\R)$.  As $\displaystyle \re\int\omega_3 dz=\re z$ is increasing on $\R$, the projection of the boundary of $f(X_1)$ onto the $x_1x_3$-plane is an embedding.

The Gauss map is given by 
\[
\fracG(z)=\frac{\im\left(\omega_2\overline{\omega}_3,-\omega_1\overline{\omega}_3,\omega_1\overline{\omega}_2\right)}{||\im\left(\omega_2\overline{\omega}_3,-\omega_1\overline{\omega}_3,\omega_1\overline{\omega}_2\right)||},
\]
and $\displaystyle -\omega_1\overline{\omega}_3=-\frac{z+\lambda_1}{w}$.  On the interior of $X_1$, we have 
\[
\im z\neq 0 \Rightarrow \im w\neq 0 \Rightarrow \im(-\omega_1\overline{\omega}_3)\neq 0.
\]
Hence, the $x_2$-coordinate of the normal vector to $f(X_1)$ never vanishes on the interior of $f(X_1)$, and so the projection of $f(X_1)$ to the $x_1x_3$-plane is a submersion that is one-to-one on the boundary.  Therefore, the projection is one-to-one on $f(X_1)$, and $f(X_1)$ is a graph over the $x_1x_3$-plane.

By Theorem, \ref{thm:order3}, $f$ is proper.  
\end{proof}

We do not know much about the existence of properly embedded tori with ends of 
type $(2,2,n)$ for $n>3$.

\subsection{Spheres with two  ends of type \texorpdfstring{$(2,2,n)$}{(2,2,n)}}
\label{sec:(2,2,n)}
An example of Karcher of a complete minimal sphere with two Enneper ends has an embedded variation with two ends of type $(2,2,3)$ is shown in Figure \ref{figure:enneper2}.

\begin{figure}[h]
	\centerline{ 
		\includegraphics[width=2.5in]{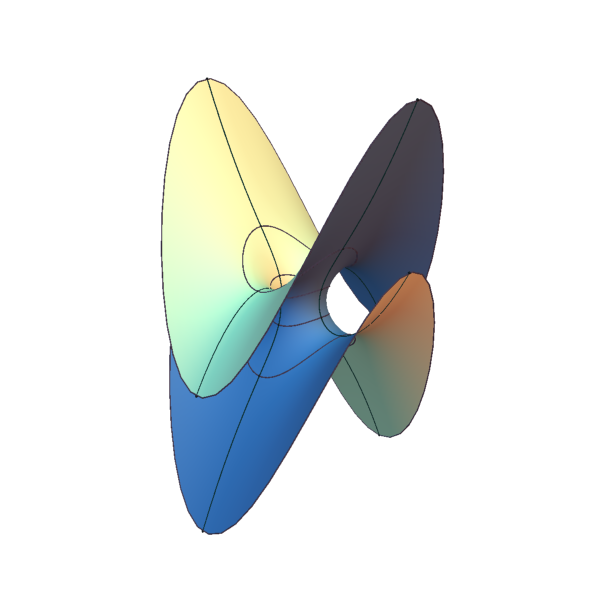}
		\includegraphics[width=2.5in]{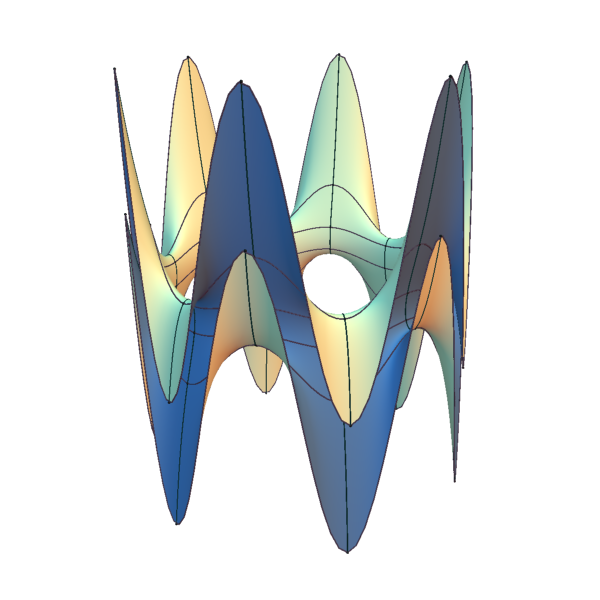}
	}
	\caption{Sphere with two $(2,2,3)$ ends and a sphere with two $(2,2,7)$ ends}
	\label{figure:enneper2}
\end{figure}

The meromorphic  1-forms are defined in $\C^*$ and given by 
\begin{align*}
\omega_1&= \left(1-\frac1{z^2}\right)\, dz
\\
\omega_2&=i  \left(1+\frac1{z^2}\right)\, dz
\\
\omega_3 &= a \left(z^{n-2}-z^{-n}\right)+b\frac n z \, dz
\end{align*}
for positive real constants $a$ and $b$.

To see that these surfaces are regular and embedded, we compute in polar coordinates $z= r e^{i t}$ that
\[
f(r,t) = \left(
\frac1r (r^2+1)\cos(t), -\frac1r (r^2+1)\sin(t), 
\frac a {n-1} (r^{1-n}+r^{n-1})\cos((n-1)t) +b n \log(r)
\right)\ .
\]
For regularity, it is easy to check that the third coordinate of $f_r\times f_t$ is equal to $1/r^3-r$.

If we had $f(r_1,t_1) = f(r_2, t_2)$, the first two coordinates would imply that $t_1=t_2 \pmod{2\pi}$ and either $r_1=r_2$ or $r_2=1/r_1$. In the latter case, the last coordinate would imply that $b n \log(r_1)=-b n \log(r_1)$ as the first term is invariant under the $r\mapsto 1/r$. That, however, forces $r_1=r_2$.

\subsection{Several ends of type \texorpdfstring{$(0,0,1)$}{(0,0,1)}}
\label{subsection:(0,0,1)}
The simplest type of end is the horn end of type $(0,0,1)$. It is very easy to construct symmetric embedded examples.

Note, however, that the direction of these ends is not completely arbitrary, as the maximum principle has to be satisfied. Therefore one needs to have at least four such ends, as shown in Figure \ref{figure:4(0,0,1)}.  Figure \ref{figure:urchin} shows a symmetric example with more ends that is easily seen to be embedded.

\begin{figure}[h]
	\centerline{ 
		\includegraphics[width=2.5in]{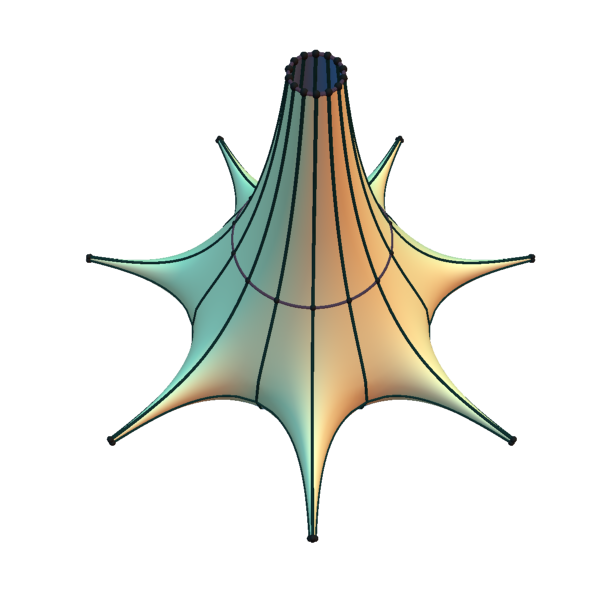}
	}
	\caption{A symmetric  9-cusp with the bottom cusp hidden}
	\label{figure:urchin}
\end{figure}

Note that the residue theorem still requires the flux vectors of the ends to balance.
 
The coordinate 1-forms for the symmetric $k+2$-cusp are defined on the plane punctured at the $k^{\text{th}}$ roots and unity, 0, and $\infty$, and are given by
\begin{align*}
\omega_1&= \frac {z^{k-2}+1}{z^k-1}\, dz
\\
\omega_2&=i \frac {z^{k-2}-1}{z^k-1}\, dz
\\
\omega_3 &= \frac1 z\, dz
\end{align*}

This surface is embedded for $k\geq 3$ - if $k<3$ then the maximum principle is violated - and has total curvature $\displaystyle\int_{\C}K dA=-2k\pi$.

When $k=4$, the six ends are asymptotic to the coordinate axes.  If we translate a copy of this surface by the vector $(n,n,n)$ then, for sufficiently large $n$, this surface doesn't intersect the original surface.  Thus, the Strong Half-Space theorem of Hoffman and Meeks doesn't hold for harmonic surfaces.

\subsection{Adding ends of type \texorpdfstring{$(0,0,1)$}{(0,0,1)} to a surface}
Since ends of type (0,0,1) occupy very little space, it is quite easy to modify embedded harmonic surfaces by adding any number of (0,0,1) ends.  A nice example is adding one end of type (0,0,1) to the sphere with a single (2,3,4) end.  Let

\begin{align*} 
   \omega_1 = {}& i \\
    \omega_2 = {}&  i z-\frac 1 z\\
     \omega_3 = {}& z^2+1 \\
\end{align*}

so that
\[
f(x,y) = \left(-y, -x y - \frac12\log\left(x^2+y^2\right), \frac13 x\left(x^2-3y^2+1\right)\right)
\]

To show that this is an embedding, it suffices to show that for fixed $y$ the curves
\[
x\mapsto c(x) =\left(-x y - \frac12\log\left(x^2+y^2\right), \frac13 x\left(x^2-3y^2+1\right)\right)
\]
are embeddings.
We compute
\[
c'(x) = \left(-\frac{x}{x^2+y^2}-y,x^2-y^2+1\right)
\]

From this it is easy to see that 
\begin{enumerate}
\item if $y^2<1$, then $c'(x)$ is never horizontal;
\item if $4 y^4>1$, then $c'(x)$ is never vertical
\end{enumerate}

This implies both that $f$ is injective and $df$ has full rank everywhere.

\begin{figure}[h]
	\centerline{ 
		\includegraphics[width=2.5in]{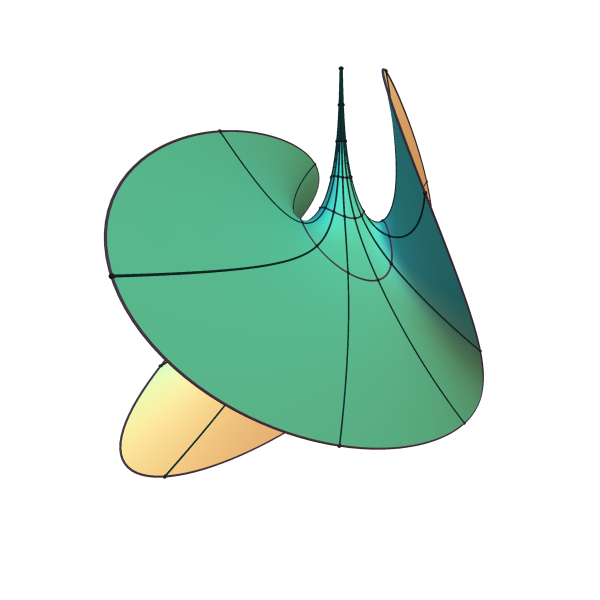}
	}
	\caption{Embedded 2-Noid with one horn and one $(2,3,4)$-end}
	\label{figure:horn234}
\end{figure}

\subsection{Spheres with  ends of type \texorpdfstring{$(0,1,2)$}{(0,1,2)}}

Ends of type $(0,1,2)$ occupy very little space --- therefore many arrangements are possible.

\begin{figure}[h]
	\centerline{ 
		\includegraphics[height=2.5in]{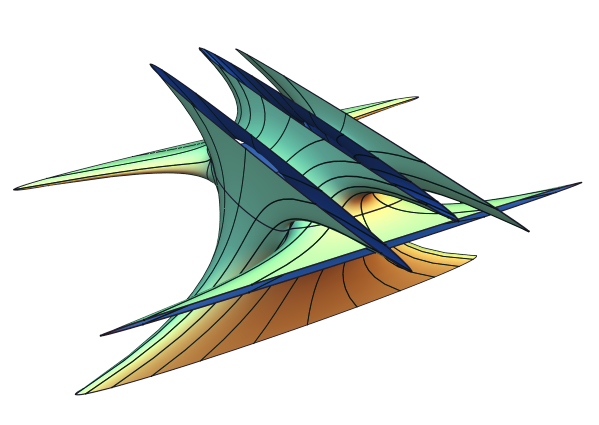}
		\includegraphics[width=2.8in]{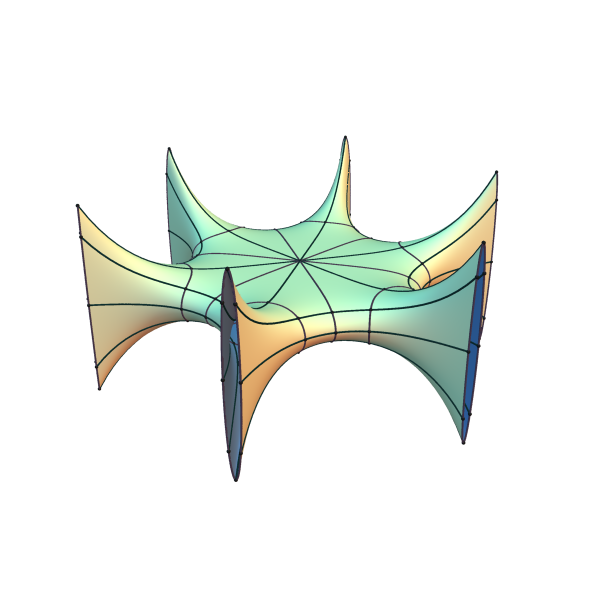}
	}
	\caption{Two spheres with six (0,1,2) ends}
	\label{figure:6(0,1,2)}
\end{figure}


As with (0,0,1) ends, one can construct an embedded sphere with k (0,1,2) ends placed at the $k^{\text{th}}$ roots of unity .  
\begin{proposition}
Let
\[
X=\bar{\C}-\{1,e^{2\pi i/k},e^{4\pi i/k},\cdots,e^{2(k-1)\pi i/k}\},
\] 
\begin{equation}
\omega_1=\frac{z^{k-2}+1}{k(z^k-1)}dz,\hspace{.2in} \omega_2=i\frac{z^{k-2}-1}{k(z^k-1)}dz,\hspace{.2in} \omega_3=\frac{z^{k-1}}{(z^k-1)^2}dz
\end{equation}
\[
\Omega=(\omega_1,\omega_2,\omega_3).
\]
The map
\begin{equation}
\displaystyle f(z)=\re\int^z\Omega
\label{eq1}
\end{equation}
is a proper harmonic embedding of $X$ into $\R^3$.
\end{proposition}

These surfaces are defined on the punctured sphere with ends of type  $(0,1,2)$ at the punctures corresponding to the k-th roots of unity.  The surface has a rotational symmetry of $\theta=e^{2\pi i/k}$.  

The ends of the surface are located at the k-th roots of unity.  Due to the rotational symmetry of the surface, each end is the same type.  The end at $z=1$ is type $(1,0,2)$, and so all the other ends are of the same type.

\begin{lemma}
The map given by equation \ref{eq1} is an immersion.
\end{lemma}

\begin{proof}
In order for $f$ to be an immersion, we need 
\[
f_x\times f_y=\im\left(\omega_2\bar{\omega}_3,-\omega_1\bar{\omega_3},\omega_1\bar{\omega}_2\right)\neq \vec{0}
\]
If $z=0$ then $\displaystyle\omega_1\bar{\omega}_2=\frac{i}{k^2}$.  If $z=\infty$ then $\displaystyle\omega_1\bar{\omega}_2=\frac{i}{k^2}$.  If $z\not\in\{0, \infty\}$ then
\[
\im\left(\omega_2\bar{\omega}_3,-\omega_1\bar{\omega_3},\omega_1\bar{\omega}_2\right)=\left(\re\left(\frac{(z^{k-2}-1)\bar{z}^{k-1}}{k|z^k-1|^2(\bar{z}^k-1)}\right),-\im\left(\frac{(z^{k-2}+1)\bar{z}^{k-1}}{k|z^k-1|^2(\bar{z}^k-1)}\right),\frac{1-|z^{k-2}|^2}{k^2|z^k-1|^2}\right)
\]
Thus, $\im (\omega_1\bar{\omega}_2)=0$ iff $z=e^{it}$ for some $t\in[0,2\pi]$.  If $z=e^{it}$ then

\[
\im\left(\omega_2\bar{\omega}_3,-\omega_1\bar{\omega_3},\omega_1\bar{\omega}_2\right)=\left(\frac{2(\cos{((k-1)t)}-\cos{t})}{k|e^{kti}-1|^4},-\frac{2(\sin{((k-1)t)}+\sin{t})}{k|e^{kti}-1|^4},0\right)
\]

However, 
\[
\left(2(\cos{((k-1)t)}-\cos{t}),-2(\sin{((k-1)t)}+\sin{t})\right)=(0,0)
\]
iff $z=e^{2\pi ji/k}$, but $e^{2\pi ji/k}\not\in X$.  Thus, $f$ is an immersion.
\end{proof}

\begin{lemma}
The surface given by equation \ref{eq1} is invariant under reflection in two orthogonal planes and a rotation about the $x_3$-axis by $e^{2\pi i/k}$.
\label{lm}
\end{lemma}

\begin{proof}
Let $\alpha:\bar{C}\mapsto\bar{C}$, $\tau:\bar{C}\mapsto\bar{C}$, and $\sigma:\bar{C}\mapsto\bar{C}$ be given by $\alpha(z)=\bar{z}$, $\displaystyle\tau(z)=\frac{1}{\bar{z}}$, and $\displaystyle\sigma(z)=e^{2\pi i/k}z$.  Then it is easy to show that 
\[
\alpha^*(\omega_1,\omega_2,\omega_3)=(\bar{\omega}_1,-\bar{\omega}_2,\bar{\omega}_3),
\]
\[
\tau^*(\omega_1,\omega_2,\omega_3)=(\bar{\omega}_1,\bar{\omega}_2,-\bar{\omega}_3),
\]
and
\[
\sigma^*(\omega_1+\omega_2i,\omega_3)=(e^{2\pi i/k}(\omega_1+\omega_2i),\omega_3).
\]

Thus, $\alpha$ induces a reflection in a plane parallel to the $(x_1,x_3)$ plane, $\tau$ induces a reflection in a plane parallel to the $(x_1,x_2)$ plane, and $\sigma$ induces a rotation about the $x_3$-axis by $e^{2\pi i/k}$.

The fixed-point sets of the reflections are the images of the fixed point sets of $\alpha$, $\tau$, and $\sigma$.  The fixed point set of $\alpha$ is $\R$.  The fixed point set of $\tau$ is $\{z :\hspace{.1cm} |z|=1\}$.  The image of the fixed point set of $\tau$ is a curve in the horizontal plane $\displaystyle x_3=\frac{1}{2k}$ because 
\[
\begin{split}
f_3(e^{it})&=\re\frac{-1}{k(e^{ikt}-1)}\\
&=\re\frac{1-e^{-ikt}}{k|e^{ikt}-1|^2}\\
&=\frac{1-\cos(kt)}{k\left(\left(\cos(kt)-1\right))^2+\sin^2(kt)\right)}\\
&=\frac{1-\cos(kt)}{k\left(\cos^2(kt)-2\cos(kt)+1+\sin^2(kt)\right)}\\
&=\frac{1-\cos(kt)}{k\left(2-2\cos(kt)\right)}\\
&=\frac{1}{2k}\\
\end{split}
\]

\end{proof}

\begin{lemma}
$f(X)$ is a properly embedded surface.
\end{lemma}

\begin{proof}
By Lemma \ref{lm}, $f(X)$ can be split into two pieces bounded by the symmetry curve given by the fixed point set of $\tau$.  We will show that $f(X)$ is an embedding by showing that each congruent piece is a graph over a plane.

Let $X_1=X-\{z :|z|<1\}$ and $X_2=X-\{z: |z|\geq 1\}$.

On the interior of $X_1$, we have 
\[
|z|<1 \Rightarrow \im(\omega_1\overline{\omega}_2) > 0.
\]
Hence, the normal vector to $f(X_1)$ always points up on the interior of $f(X_1)$, and so the $f(X_1)$ is a graph over the unit disk.  

If $z=re^{i t}$ and $r<1$ then
\[
\begin{split}
f_3(re^{it})&=\re\frac{-1}{k(r^ke^{ikt}-1)}\\
&=\re\frac{1-r^ke^{-ikt}}{k|r^ke^{ikt}-1|^2}\\
&=\frac{1-r^k\cos(kt)}{k\left(\left(r^k\cos(kt)-1\right)^2+r^{2k}\sin^2(kt)\right)}\\
&=\frac{1-r^k\cos(kt)}{k\left(r^{2k}\cos^2(kt)-2r^k\cos(kt)+1+r^{2k}\sin^2(kt)\right)}\\
&=\frac{1-r^k\cos(kt)}{k\left(r^{2k}-2r^k\cos(kt)+1\right)}\\
&\geq\frac{1}{2k}\\
\end{split}
\]
Hence, $f(X_1)$ stays above the symmetry plane $\displaystyle x_3=\frac{1}{2k}$, and so $f(X_1)$ and $f(X_2)$ are both embedded and disjoint.  Thus, $f(X)$ is embedded.

As $|z|\rightarrow 1$, $|f_1(z)|\approx|\re\log{|z-1|}|\rightarrow\infty$.  Together with the symmetries of the surface, this proves that $f$ is proper.  
\end{proof}

\section{Open Questions}
\label{sec:open}

This section gathers some open questions we found interesting, arranged by topic.

\subsection{Construction Problems}

Most of the surfaces an ends found in this paper were discovered through trial and error, guided by intuition we had learned
from minimal surfaces. For example, to construct  an embedded end of a given type, one usually has to add lower order terms to the relevant forms to make sure that the end is proper, complete, and embedded. Finding these lower order terms has been delicate and on a case-by-case basis.

\begin{problem}
More concretely, for which types do there exist complete, properly embedded ends (or surfaces with an end of that type)?
 \end{problem}
 
\begin{problem}
We have several examples of properly embedded ends for which we could not find complete surfaces with such ends. This raises the more general question for which collections of ends one can find a properly embedded complete harmonic surface with these ends. Negative results would be interesting as well.
\end{problem}
 
 \begin{problem}
If an embedded harmonic surface exists, is there one with higher (or lower) genus and the same end types? For instance, 
is there an embedded harmonic torus with a single end of type $(2,2,n)$ for $n>3$? Or, is there an embedded torus with a single end of type $(2,3,4)$?.
\end{problem}

 \begin{problem}
We have an embedded example of  a genus two surface with three  ends of type $(2,2,3)$. Is such a surface also possible  with lower genus?
\end{problem}

\subsection{Conceptual Questions}

\begin{problem}
Can one intrinsically characterize what Riemannian surfaces admit complete  and proper harmonic parametrizations (of some codimension) in our restricted sense?
\end{problem}

\begin{problem}
Is there a natural larger class of Riemannian surfaces that satisfy our Gauss-Bonnet formula with a suitable notion of order of an end?
\end{problem}

\begin{problem}
What can one say about the moduli space of embedded harmonic surfaces of fixed genus and fixed end types?
For instance,
is the space  of complete, properly embedded punctured spheres with $n$  ends  of a given type (say $(0,1,2)$ or $(0,2,3)$)  connected? See \cite{al1} for a classification of harmonic catenoids.
\end{problem}

\begin{problem}
For embedded minimal surfaces, the  dimension of the moduli space at non-degenerate surfaces is determined by the number of ends and independent of the genus.
What is the situation for harmonic surfaces?
 \end{problem}
 
\begin{problem}
Is there an analogue to Osserman's Theorem? More precisely, when is a complete harmonic surface of finite total curvature defined on a punctured compact Riemann surface by data that extend meromorphically into the punctures?
\end{problem}

\begin{problem}
Can one classify all harmonic tori with total curvature $-8\pi$? 
\end{problem}

\bibliographystyle{plain}
\bibliography{bibliography}

\end{document}